\documentclass[a4paper,10pt]{article}
\usepackage{amssymb}
\usepackage[dvips]{color}
\usepackage{graphicx}
\usepackage{amsmath}
\usepackage{a4wide}

\newtheorem{theorem}{Theorem}

\newtheorem{corollary}[theorem]{Corollary}

\newtheorem{definition}[theorem]{Definition}
\newtheorem{example}[theorem]{Example}

\newtheorem{lemma}[theorem]{Lemma}
\newtheorem{notation}[theorem]{Notation}

\newtheorem{proposition}[theorem]{Proposition}
\newtheorem{remark}[theorem]{Remark}

\newenvironment{proof}[1][Proof]{\textbf{#1.} }{\ \rule{0.5em}{0.5em}}

\begin{document}

\title{Conditionally Evenly Convex Sets and Evenly Quasi-Convex Maps}
\author{Marco Frittelli\thanks{%
Dipartimento di Matematica, Universit\`{a} degli Studi di Milano.} \and %
Marco Maggis\thanks{%
Dipartimento di Matematica Universit\`{a} degli Studi di Milano.}}
\maketitle

\begin{abstract}
Evenly convex sets in a topological vector space are defined as the
intersection of a family of open half spaces. We introduce a generalization
of this concept in the conditional framework and provide a generalized
version of the bipolar theorem. This notion is then applied to obtain the
dual representation of conditionally evenly quasi-convex maps.
\end{abstract}

\section{Introduction}

A subset $C$ of a topological vector space is \emph{evenly convex} if it is
the intersection of a family of open half spaces, or equivalently, if every $%
x\notin C$ can be separated from $C$ by a continuous linear functional.
Obviously an evenly convex set is necessarily convex. This idea was firstly
introduced by Fenchel \cite{Fe52} aimed to determine the largest family of
convex sets $C$ for which the polarity $C=C^{00}$ holds true. It is well
known that in the framework of incomplete financial markets the Bipolar
Theorem is a key ingredient when we represent the super replication price of
a contingent claim in terms of the class of martingale measures. Recently
evenly convex sets and in particular evenly quasi-concave real valued
functions have been considered by Cerreia-Vioglio, Maccheroni, Marinacci and
Montrucchio in the context of Decision Theory \cite{CV09} and Risk Measures 
\cite{CMMMa}. Evenly quasiconcavity is the weakest notion that enables, in
the static setting, a \textit{complete} quasi-concave duality, which is a
key structural property regarding the dual representation of the behavioral
preferences and Risk Measures. Similarly Drapeau and Kupper \cite{DK10}
obtained a complete static quasi-convex duality under slightly different
conditions of the risk preferences structure that is strictly related to the
notion of evenly convexity.

\bigskip

In a conditional framework, as for example when $\mathcal{F}$ is a sigma
algebra containing the sigma algebra $\mathcal{G}$ and we deal with $%
\mathcal{G}$-conditional expectation, $\mathcal{G}$-conditional sublinear
expectation, $\mathcal{G}$-conditional risk measure, the analysis of the
duality theory is more delicate. We may consider conditional maps $\rho
:E\rightarrow L^{0}(\Omega ,\mathcal{G},\mathbb{P})$ defined either on
vector spaces (i.e. $E=L^{p}(\Omega ,\mathcal{F},\mathbb{P})$) or on $L^{0}$%
-modules (i.e. $E=L_{\mathcal{G}}^{p}(\mathcal{F}):=\left\{ yx\mid y\in
L^{0}(\Omega ,\mathcal{G},\mathbb{P})\text{ and }x\in L^{p}(\Omega ,\mathcal{%
F},\mathbb{P})\right\} $).

As described in details by Filipovic, Kupper and Vogelpoth \cite{FKV09}, 
\cite{FKV10} and by Guo \cite{Guo} the $L^{0}$-modules approach (see also
Section 3 for more details) is a very powerful tool for the analysis of
conditional maps and their dual representation.

\bigskip

In this paper we show that in order to achieve a conditional version of the
representation of evenly quasi-convex maps a good notion of evenly convexity
is crucial. We introduce the concept of a \emph{conditionally evenly convex
set, }which is tailor made for the conditional setting, in a framework that
exceeds the module setting alone, so that will be applicable in many
different context. \newline
In Section \ref{set} we provide the characterization of evenly convexity
(Theorem \ref{main} and Proposition \ref{prop24}) and state the conditional
version of the Bipolar Theorem (Theorem \ref{bipolar}). Under additional
topological assumptions, we show that conditionally convex sets that are
closed or open are conditionally evenly convex (see Section \ref{module&set}%
, Proposition \ref{LLSC}). As a consequence, the conditional evenly
quasiconvexity of a function, i.e. the property that the conditional lower
level sets are evenly convex, is a weaker assumption than quasiconvexity and
lower (or upper) semicontinuity.

\bigskip

In Section \ref{map} we apply the notion of conditionally evenly convex set
to the \textit{the dual representation of evenly quasiconvex maps, }i.e.%
\textit{\ }conditional maps $\rho :E\rightarrow L^{0}(\Omega ,\mathcal{G},%
\mathbb{P}) $ with the property that the conditional lower level sets are
evenly convex. We prove in Theorem \ref{EvQco} that an evenly quasiconvex
regular map $\pi:E\rightarrow \bar{L}^{0}(\mathcal{G})$ can be represented
as 
\begin{equation}
\pi (X)=\sup_{\mu \in \mathcal{L}(E,L^{0}(\mathcal{G}))}\mathcal{R}(\mu
(X),\mu ),  \label{rr}
\end{equation}%
where%
\begin{equation*}
\mathcal{R}(Y,\mu ):=\inf_{\xi \in E}\left\{ \pi (\xi )\mid \mu (\xi )\geq
Y\right\} ,\text{ }Y\in L^{0}(\mathcal{G}), 
\end{equation*}%
$E$ is a topological $L^{0}$-module and $\mathcal{L}(E,L^{0}(\mathcal{G}))$
is the module of continuous $L^{0}$-linear functionals over $E$.

\bigskip

The proof of this result is based on a version of the hyperplane separation
theorem and not on some approximation or scalarization arguments, as it
happened in the vector space setting (see \cite{FM09}). By carefully
analyzing the proof one may appreciate many similarities with the original
demonstration in the static setting by Penot and Volle \cite{PV}. One key
difference with \cite{PV}, in addition to the conditional setting, is the
continuity assumption needed to obtain the representation (\ref{rr}). We
work, as in \cite{CV09}, with evenly quasiconvex functions, an assumption
weaker than quasiconvexity and lower (or upper) semicontinuity. \newline
As explained in \cite{FM09} the representation of the type (\ref{rr}) is a
cornerstone in order to reach a robust representation of Quasi-convex Risk
Measures or Acceptability Indexes.

\section{On Conditionally Evenly Convex sets}

\label{set}

The probability space $(\Omega ,\mathcal{G},\mathbb{P})$ is fixed throughout
this paper. Whenever we will discuss conditional properties we will always
make reference, even without explicitly mentioning it in the notations - to
conditioning with respect to the sigma algebra $\mathcal{G}$.

We denote with $L^{0}=:L^{0}(\Omega ,\mathcal{G},\mathbb{P})$ the space of $%
\mathcal{G}$ measurable random variables that are $\mathbb{P}$ a.s. finite,
whereas by $\bar{L}^{0}$ the space of extended random variables which may
take values in $\mathbb{R}\cup \{\infty \}$. We remind that all
equalities/inequalities among random variables are meant to hold $\mathbb{P}$%
-a.s.. As the expected value $E_{\mathbb{P}}[\cdot ]$ is mostly computed
w.r.t. the reference probability $\mathbb{P}$, we will often omit $\mathbb{P}
$ in the notation. For any $A\in \mathcal{G}$ the element $\mathbf{1}_{A}\in
L^{0}$ is the random variable a.s. equal to $1$ on $A$ and $0$ elsewhere. In
general since $(\Omega ,\mathcal{G},\mathbb{P})$ are fixed we will always
omit them. We define $L_{+}^{0}=\{X\in L^{0}\mid X\geq 0\}$ and $%
L_{++}^{0}=\{X\in L^{0}\mid X>0\}$. \newline
The essential ($\mathbb{P}$ almost surely) \emph{supremum} $ess\sup_{\lambda
}(X_{\lambda })$ of an arbitrary family of random variables $X_{\lambda }\in
L^{0}(\Omega ,\mathcal{F},\mathbb{P})$ will be simply denoted by $%
\sup_{\lambda }(X_{\lambda })$, and similarly for the essential \emph{infimum%
} (see \cite{FoSch} Section A.5 for reference).

\begin{definition}[Dual pair]
\label{dualP}
\end{definition}

A dual pair $(E,E^{\prime },\langle \cdot ,\cdot \rangle )$ consists of:

\begin{enumerate}
\item $(E,+)$ (resp. $(E^{\prime },+)$) is any structure such that the
formal sum $x\mathbf{1}_{A}+y\mathbf{1}_{A^{C}}$ belongs to $E$ (resp. $%
x^{\prime }\mathbf{1}_{A}+y^{\prime }\mathbf{1}_{A^{C}}\in E^{\prime }$) for
any $x,y\in E$ (resp. $x^{\prime },y^{\prime }\in E^{\prime }$) and $A\in 
\mathcal{G}$ with $\mathbb{P}(A)>0$ and there exists an null element $0\in E$
(resp. $0\in E^{\prime }$) such that $x+0=x$ for all $x\in E$ (resp. $%
x^{\prime }+0=x^{\prime }$ for all $x^{\prime }\in E^{\prime }$).

\item A map $\langle \cdot ,\cdot \rangle :E\times E^{\prime }\rightarrow
L^{0}$ such that 
\begin{eqnarray*}
\langle x\mathbf{1}_{A}+y\mathbf{1}_{A^{C}},x^{\prime }\rangle &=&\langle
x,x^{\prime }\rangle \mathbf{1}_{A}+\langle y,x^{\prime }\rangle \mathbf{1}%
_{A^{C}} \\
\langle x,x^{\prime }\mathbf{1}_{A}+y^{\prime }\mathbf{1}_{A^{C}}\rangle
&=&\langle x,x^{\prime }\rangle \mathbf{1}_{A}+\langle x,y^{\prime }\rangle 
\mathbf{1}_{A^{C}} \\
\langle 0,x^{\prime }\rangle =0 &\text{and}&\langle x,0\rangle =0
\end{eqnarray*}%
for every $A\in \mathcal{G}$, $\mathbb{P}(A)>0$ and $x,y\in E$, $x^{\prime
},y^{\prime }\in E^{\prime }$.
\end{enumerate}

Clearly in many applications $E$ will be a class of random variables (as
vector lattices, or $L^{0}$-modules as in the Examples \ref{ex3} and \ref%
{exmod}) and $E^{\prime }$ is a selection of conditional maps, for example
conditional expectations, sublinear conditional expectations, conditional
risk measures.

We recall from \cite{FKV09} an important type of concatenation:

\begin{definition}[Countable Concatenation Hull]
.

\begin{description}
\item[(CSet)] A subset $\mathcal{C}\subset E$ has the countable
concatenation property if for every countable partition $\{A_{n}\}_{n}%
\subseteq \mathcal{G}$ and for every countable collection of elements $%
\{x_n\}_n\subset \mathcal{C}$ we have $\sum_{n}\mathbf{1}_{A_{n}}x_n\in 
\mathcal{C}$.
\end{description}

Given $\mathcal{C\subseteq }E$, we denote by $\mathcal{C}^{cc}$ the
countable concatenation hull of $\mathcal{C}$, namely the smallest set $%
\mathcal{C}^{cc}\supseteq \mathcal{C}$ which satisfies (CSet): 
\begin{equation*}
\mathcal{C}^{cc}=\left\{ \sum_{n}\mathbf{1}_{A_{n}}x_{n}\mid x_{n}\in 
\mathcal{C}\text{, }\{A_{n}\}_{n}\subseteq \mathcal{G}\text{ is a partition
of }\Omega \right\} . 
\end{equation*}%
These definitions can be plainly adapted to subsets of $E^{\prime }$. 
\newline
The action of an element $\xi ^{\prime }=\sum_{m}\mathbf{1}%
_{B_{m}}x_{m}^{\prime }\in (E^{\prime })^{cc}$ over $\xi =\sum_{n}\mathbf{1}%
_{A_{n}}x_{n}\in E^{cc}$ is defined as 
\begin{equation}
\langle \xi ,\xi ^{\prime }\rangle =\left\langle \sum_{n}\mathbf{1}%
_{A_{n}}x_{n},\sum_{m}\mathbf{1}_{B_{m}}x_{m}^{\prime }\right\rangle
=\sum_{n}\sum_{m}\langle x_{n},x_{m}^{\prime }\rangle \mathbf{1}_{A_{n}\cap
B_{m}}  \label{FuncExt}
\end{equation}%
and does not depend on the representation of $\xi ^{\prime }\in (E^{\prime
})^{cc}$ and $\xi \in \mathcal{C}^{cc}$.
\end{definition}

\begin{example}
\label{ex3}Let $\mathcal{F}$ be a sigma algebra containing $\mathcal{G}$.
Consider the vector space $E:=L^{p}(\mathcal{F}):=L^{p}(\Omega ,\mathcal{F},%
\mathbb{P})$, for $p\geq 1$. If we compute the countable concatenation hull
of $L^{p}(\mathcal{F})$ we obtain exactly the $L^{0}$-module 
\begin{equation*}
L_{\mathcal{G}}^{p}(\mathcal{F}):=\left\{ yx\mid y\in L^{0}(\mathcal{G})%
\text{ and }x\in L^{p}(\mathcal{F})\right\} 
\end{equation*}%
as introduced in \cite{FKV09} and \cite{FKV10} (see Example \ref{exmod} for
more details).

Similarly, the class of conditional expectations $\mathcal{E}=\{E[\;\cdot Z|%
\mathcal{G}]\mid Z\in L^{q}(\Omega ,\mathcal{F},\mathbb{P})\}$ and $\frac{1}{%
p}+\frac{1}{q}=1$ can be identified with the space $L^q(\mathcal{F} )$.
Hence the countable concatenation hull $\mathcal{E}^{cc}$ will be exactly $%
L^q_{\mathcal{G} }(\mathcal{F} )$, the dual $L^{0}$-module of $L_{\mathcal{G}%
}^{p}(\mathcal{F})$.
\end{example}

If $E$ (or $E^{\prime })$ does not fulfill (CSet) we can always embed the
theory in its concatenation hull and henceforth we make the following:

\bigskip

\noindent \textbf{Assumption:} In the sequel of this paper we always suppose
that both $E$ and $E^{\prime }$ satisfies (CSet).

\bigskip

We recall that a subset $C$ of a locally convex topological vector space $V$
is \emph{evenly convex} if it is the intersection of a family of open half
spaces, or equivalently, if every $x\notin C$ can be separated from $C$ by a
continuous real valued linear functional. As the intersection of an empty
family of half spaces is the entire space $V$, the whole space $V$ itself is
evenly convex.

However, in order to introduce the concept of conditional evenly convex set
(with respect to $\mathcal{G}$) we need to take care of the fact that the
set $C$ may present some components which degenerate to the entire $E$.
Basically it might occur that for some $A\in \mathcal{G}$ 
\begin{equation*}
C\mathbf{1}_{A}=E\mathbf{1}_{A}, 
\end{equation*}%
i.e., for each $x\in E$ there exists $\xi \in C$ such that $\xi \mathbf{1}%
_{A}=x\mathbf{1}_{A}$. In this case there are no chances of finding an $x\in
E$ satisfying $\mathbf{1}_{A}C\cap \mathbf{1}_{A}\{x\}=\varnothing $ and
consequently no conditional separation may occur. It is clear that the
evenly convexity property of a set $C$ is meaningful only on the set where $%
C $ does not coincide with the entire $E$. Thus we need to determine the
maximal $\mathcal{G}$-measurable set on which $C$ reduces to $E$. To this
end, we set the following notation that will be employed many times.

\begin{notation}
\label{trivialcomponent}Fix a set $\mathcal{C}\subseteq E$. As the class $%
\mathcal{A}(\mathcal{C}):=\{A\in \mathcal{G}\mid \mathcal{C}\mathbf{1}_{A}=E%
\mathbf{1}_{A}\}$ is closed with respect to countable union, we denote with $%
A_{\mathcal{C}}$ the $\mathcal{G}$-measurable maximal element of the class $%
\mathcal{A}(\mathcal{C})$ and with $D_{\mathcal{C}}$ the (P-a.s. unique)
complement of $A_{\mathcal{C}}$ (see also the Remark \ref{remMAX}). Hence $%
\mathcal{C}\mathbf{1}_{A_{\mathcal{C}}}=E\mathbf{1}_{A_{\mathcal{C}}}.$
\end{notation}

We now give the formal definition of conditionally evenly convex set in
terms of intersections of hyperplanes in the same spirit of \cite{Fe52}.

\begin{definition}
A set $\mathcal{C}\subseteq E$ is conditionally evenly convex if there exist 
$\mathcal{L}\subseteq E^{\prime }$ (in general non-unique and empty if $%
\mathcal{C}=E$) such that 
\begin{equation}
\mathcal{C}=\bigcap_{x^{\prime }\in \mathcal{L}}\{x\in E\mid \langle
x,x^{\prime }\rangle <Y_{x^{\prime }}\text{ on }D_{\mathcal{C}}\}\quad \text{%
for some }Y_{x^{\prime }}\in L^{0}.  \label{definition}
\end{equation}
\end{definition}

\begin{remark}
Notice that for any arbitrary $D\in\mathcal{G} $, $\mathcal{L}\subseteq
E^{\prime }$ the set 
\begin{equation*}
\mathcal{C}= \bigcap_{x^{\prime }\in \mathcal{L}}\{x\in E\mid \langle
x,x^{\prime }\rangle <Y_{x^{\prime }}\text{ on }D \} \quad \text{ for some }
Y_{x^{\prime }}\in L^{0} 
\end{equation*}
is evenly convex, even though in general $D_{\mathcal{C}}\subseteq D$.
\end{remark}

\begin{remark}
\label{remark} We observe that since $E$ satisfies (CSet) then automatically
any conditionally evenly convex set satisfies (CSet). As a consequence there
might exist a set $\mathcal{C}$ which fails to be conditionally evenly
convex, since does not satisfy (CSet), but $\mathcal{C}^{cc}$ is
conditionally evenly convex. Consider for instance $E=L^1_{\mathcal{G} }(%
\mathcal{F} ), E^{\prime}=L^{\infty}_{\mathcal{G} }(\mathcal{F} )$, endowed
with the pairing $\langle x,x^{\prime}\rangle=E[xx^{\prime} | \mathcal{G} ]$%
. Fix $x^{\prime}\in L^{\infty}(\mathcal{F} ), Y\in L^0(\mathcal{G} )$ and
the set 
\begin{equation*}
\mathcal{C}=\{x\in L^1(\mathcal{F} ) \mid E[xx^{\prime}\mid \mathcal{G} ]<
Y\}. 
\end{equation*}
Clearly $\mathcal{C}$ is not conditionally evenly convex since $\mathcal{C}%
\subsetneqq \mathcal{C}^{cc}$; on the other hand 
\begin{equation*}
\mathcal{C}^{cc}=\{x\in L^1_{\mathcal{G} }(\mathcal{F} ) \mid E[xx^{\prime} |%
\mathcal{G} ]< Y\} 
\end{equation*}
which is by definition evenly convex.
\end{remark}

\begin{remark}
\label{convexity} Recall that a set $C\subseteq E$ is $L^{0}$\emph{-convex}
if $\Lambda x+(1-\Lambda )y\in C$ for any $x,y\in C$ and $\Lambda \in L^{0}$
with $0\leq \Lambda \leq 1$. \newline
Suppose that all the elements $x^{\prime }\in E^{\prime }$ satisfy: 
\begin{equation*}
\langle \Lambda x+(1-\Lambda )y,x^{\prime }\rangle \leq \Lambda \langle
x,x^{\prime }\rangle +(1-\Lambda )\langle y,x^{\prime }\rangle \text{, for
all }x,y\in E\text{, }\Lambda \in L^{0}\text{: }0\leq \Lambda \leq 1. 
\end{equation*}%
If $E$ is $L^{0}-$convex then every conditionally evenly convex set is also $%
L^{0}-$convex.
\end{remark}

In order to separate one point $x\in E$ from a set $C\subseteq E$ in a
conditional way we need the following definition:

\begin{definition}
For $x\in E$ and a subset $\mathcal{C}$ of $E,$ we say that $x$ is \emph{%
outside} $\mathcal{C}$ if $\mathbf{1}_{A}\{x\}\cap \mathbf{1}_{A}\mathcal{C}%
=\varnothing $ for every $A\in \mathcal{G}$ with $A\subseteq D_{\mathcal{C}}$
and $\mathbb{P}(A)>0$.
\end{definition}

This is of course a much stronger requirement than $x\notin \mathcal{C}$.

\begin{definition}
For $\mathcal{C}\subseteq E$ we define the polar and bipolar sets as follows 
\begin{eqnarray*}
\mathcal{C}^{\circ }:\,= &&\left\{ x^{\prime }\in E^{\prime }\mid \langle
x,x^{\prime }\rangle <1\text{ on }D_{\mathcal{C}}\text{ for all }x\in 
\mathcal{C}\right\} , \\
\mathcal{C}^{\circ \circ }:\,= &&\left\{ x\in E\mid \langle x,x^{\prime
}\rangle <1\text{ on }D_{\mathcal{C}}\text{ for all }x^{\prime }\in \mathcal{%
C}^{\circ }\right\} \\
&&\,\,=\bigcap\limits_{x^{\prime }\in \mathcal{C}^{\circ }}\left\{ x\in
E\mid \langle x,x^{\prime }\rangle <1\text{ on }D_{\mathcal{C}}\right\} .
\end{eqnarray*}
\end{definition}

\noindent We now state the main results of this note about the
characterization of evenly convex sets and the Bipolar Theorem. Their proofs
are postponed to the Section 4.

\begin{theorem}
\label{main}Let $(E,E^{\prime },\langle \cdot ,\cdot \rangle )$ be a dual
pairing introduced in Definition \ref{dualP} and let $\mathcal{C}\subseteq E$%
. The following statements are equivalent:

\begin{enumerate}
\item $\mathcal{C}$ is conditionally evenly convex.

\item $\mathcal{C}$ satisfies (CSet) and for every $x$ \emph{outside} $%
\mathcal{C}$ there exists $x^{\prime }\in E^{\prime }$ such that 
\begin{equation*}
\langle \xi ,x^{\prime }\rangle <\langle x,x^{\prime }\rangle \text{ on }D_{%
\mathcal{C}},\;\forall \,\xi \in \mathcal{C}. 
\end{equation*}
\end{enumerate}
\end{theorem}

\begin{theorem}[Bipolar Theorem]
\label{bipolar}Let $(E,E^{\prime },\langle \cdot ,\cdot \rangle )$ be a dual
pairing introduced in Definition \ref{dualP} and assume in addition that the
pairing $\langle \cdot ,\cdot \rangle $ is $L^{0}$-linear in the first
component i.e. 
\begin{equation*}
\langle \alpha x+\beta y,x^{\prime }\rangle =\alpha \langle x,x^{\prime
}\rangle +\beta \langle x,x^{\prime }\rangle 
\end{equation*}%
for every $x^{\prime }\in E^{\prime }$, $x,y\in E$, $\alpha ,\beta \in L^{0}$%
. For any $\mathcal{C}\subseteq E$ such that $0\in \mathcal{C}$ we have:

\begin{enumerate}
\item $\mathcal{C}^{\circ }=\left\{ x^{\prime }\in E^{\prime }\mid \langle
x,x^{\prime }\rangle <1\text{ on }D_{\mathcal{C}}\text{ for all }x\in 
\mathcal{C}^{cc}\right\} $

\item The bipolar $\mathcal{C}^{\circ \circ }$ is a conditionally evenly
convex set containing $\mathcal{C}$.

\item The set $\mathcal{C}$ is conditionally evenly convex if and only if $%
\mathcal{C}=\mathcal{C}^{\circ \circ }$.
\end{enumerate}
\end{theorem}

Suppose that the set $\mathcal{C}\subseteq E$ is a $L^{0}$-cone, i.e. $%
\alpha x\in \mathcal{C}$ for every $x\in \mathcal{C}$ and $\alpha \in
L_{++}^{0}$. In this case, it is immediate to verify that the polar and
bipolar can be rewritten as: 
\begin{eqnarray}
\mathcal{C}^{\circ } &&\,=\left\{ x^{\prime }\in E^{\prime }\mid \langle
x,x^{\prime }\rangle \leq 0\text{ on }D_{\mathcal{C}}\text{ for all }x\in 
\mathcal{C}\right\} ,  \notag \\
\mathcal{C}^{\circ \circ } &&\,=\left\{ x\in E\mid \langle x,x^{\prime
}\rangle \leq 0\text{ on }D_{\mathcal{C}}\text{ for all }x^{\prime }\in 
\mathcal{C}^{\circ }\right\} .  \label{222}
\end{eqnarray}

\section{On Conditionally Evenly Quasi-Convex maps}

\label{map}

Here we state the dual representation of conditional evenly quasiconvex maps
of the Penot-Volle type which extends the results obtained in \cite{FM09}
for topological vector spaces. We work in the general setting outlined in
Section 2. The additional basic property that is needed is regularity.

\begin{definition}
\label{defdef1}A map $\pi :E\rightarrow \bar{L}^{0}$ is

\begin{description}
\item[(REG)] regular if for every $x_{1},x_{2}\in E$ and $A\in \mathcal{G}$, 
\begin{equation*}
\pi (x_{1}\mathbf{1}_{A}+x_{2}\mathbf{1}_{A^{C}})=\pi (x_{1})\mathbf{1}%
_{A}+\pi (x_{2})\mathbf{1}_{A^{C}}. 
\end{equation*}
\end{description}
\end{definition}

\begin{remark}
\label{remB}(On REG) It is well known that (REG) is equivalent to:%
\begin{equation*}
\pi (x\mathbf{1}_{A})\mathbf{1}_{A}=\pi (x)\mathbf{1}_{A}\text{, }\forall
A\in \mathcal{G}\text{, }\forall x\in E. 
\end{equation*}
Under the countable concatenation property it is even true that (REG) is
equivalent to countably regularity, i.e. 
\begin{equation*}
\pi (\sum_{i=1}^{\infty }x_{i}\mathbf{1}_{A_{i}})=\sum_{i=1}^{\infty }\pi
(x_{i})\mathbf{1}_{A_{i}}\text{ on \ }\cup _{i=1}^{\infty }A_{i} 
\end{equation*}%
if $x_{i}\in E$ and $\left\{ A_{i}\right\} _{i}$ is a sequence of disjoint $%
\mathcal{G}$ measurable sets. Indeed $x:=\sum_{i=1}^{\infty }x_{i}\mathbf{1}%
_{A_{i}}\in E$ and $\sum_{i=1}^{\infty }\pi (x_{i})\mathbf{1}_{A_{i}}\in 
\bar{L}^{0}$; (REG) then implies $\pi (x)\mathbf{1}_{A_{i}}=\pi (x\mathbf{1}%
_{A_{i}})\mathbf{1}_{A_{i}}=\pi (x_{i}\mathbf{1}_{A_{i}})\mathbf{1}%
_{A_{i}}=\pi (x_{i})\mathbf{1}_{A_{i}}$.
\end{remark}

Let $\pi :E\rightarrow \bar{L}^{0}$ be (REG). There might exist a set $A\in 
\mathcal{G}$ on which the map $\pi$ is infinite, in the sense that $\pi (\xi
)\mathbf{1}_{A}=+\infty\mathbf{1}_{A}$ for every $\xi \in E$. For this
reason we introduce 
\begin{equation*}
\mathcal{M}:=\{A\in \mathcal{G}\mid \pi (\xi )\mathbf{1}_{A}=+\infty\mathbf{1%
}_{A}\;\forall \,\xi \in E\}. 
\end{equation*}
Applying Lemma \ref{L10} in Appendix with $F:=\left\{ \pi (\xi )\mid \xi \in
E\right\} $ and $Y_{0}=+\infty$ we can deduce the existence of two maximal
sets $T_{\pi }\in \mathcal{G}$ and $\Upsilon _{\pi }\in \mathcal{G}$ for
which $P(T_{\pi }\cap \Upsilon _{\pi })=0$, $P(T_{\pi }\cup \Upsilon _{\pi
})=1$ and 
\begin{eqnarray}
\pi (\xi )=+\infty &\text{ on }\Upsilon _{\pi }&\text{for every }\xi \in E, 
\notag \\
\pi (\zeta)<+\infty &\text{ on }T_{\pi }&\text{ for some }\zeta\in E.
\label{888}
\end{eqnarray}

\begin{definition}
\label{EQC} A map $\pi :E\rightarrow \bar{L}^{0}(\mathcal{G})$ is

\begin{description}
\item[(QCO)] conditionally quasiconvex if $U_Y= \{\xi\in E\mid \pi(\xi)%
\mathbf{1}_{T_{\pi}}\leq Y\}$ are $L^0$-convex (according to Remark \ref%
{convexity}) for every $Y\in L^0(\mathcal{G} )$.

\item[(EQC)] conditionally evenly quasiconvex if $U_Y= \{\xi\in E\mid
\pi(\xi)\mathbf{1}_{T_{\pi}}\leq Y\}$ are conditionally evenly convex for
every $Y\in L^0(\mathcal{G} )$.
\end{description}
\end{definition}

\begin{remark}
\label{remA} For $\pi :E\rightarrow \bar{L}^{0}(\mathcal{G})$ the
quasiconvexity of $\pi $ is equivalent to the condition 
\begin{equation}
\pi (\Lambda x_{1}+(1-\Lambda )x_{2})\leq \pi (x_{1})\vee \pi (x_{2}),
\label{max}
\end{equation}
for every $x_{1},x_{2}\in E$, $\Lambda \in L^{0}(\mathcal{G})$ and $0\leq
\Lambda \leq 1$. In this case the sets $\{\xi \in E\mid \pi (\xi )1_{D}<Y\}$
are $L^{0}(\mathcal{G})$-convex for every $Y\in \bar{L}^{0}(\mathcal{G})$
and $D\in \mathcal{G}$ (This follows immediately from (\ref{max})). \newline
Moreover under the further structural property of Remark \ref{convexity} we
have that (EQC) implies (QCO). We will see in the $L^0$-modules framework
that if the map $\pi$ is either lower semicontinuous or upper semicontinuous
then the reverse implication holds true (see Proposition \ref{LLSC},
Corollary \ref{LSC} and Proposition \ref{USC}).
\end{remark}

We now state the main result of this Section.

\begin{theorem}
\label{EvQco} Let $(E,E^{\prime },\langle \cdot ,\cdot \rangle )$ be a dual
pairing introduced in Definition \ref{dualP}. If $\pi :E\rightarrow \bar{L}%
^{0}(\mathcal{G})$ is (REG) and (EQC) then 
\begin{equation}
\pi (x)=\sup_{x^{\prime} \in E^{\prime }}\mathcal{R}(\langle x,x^{\prime
}\rangle ,x^{\prime }),  \label{rapprPiSup}
\end{equation}%
where for $Y\in L^{0}(\mathcal{G})$ and $x^{\prime}$, 
\begin{equation}
\mathcal{R}(Y,x^{\prime }):=\inf_{\xi \in E}\left\{ \pi (\xi )\mid \langle
\xi ,x^{\prime }\rangle \geq Y\right\} .  \label{1212}
\end{equation}
\end{theorem}

\section{Conditional Evenly convexity in $L^{0}$- modules}

\label{module&set}

This section is inspired by the contribution given to the theory of $L^{0}$%
-modules by Filipovic et al. \cite{FKV09} on one hand and on the other to
the extended research provided by Guo from 1992 until today (see the
references in \cite{Guo}).

The following Proposition \ref{LLSC} shows that the definition of a
conditionally evenly convex set is the appropriate generalization, in the
context of topological $L^{0}$ module, of the notion of an evenly convex
subset of a topological vector space, as in both setting convex (resp. $%
L^{0} $-convex) sets that are either closed or open are evenly (resp.
conditionally evenly) convex. This is a key result that allows to show that
the assumption (EQC) is the weakest that allows to reach a dual
representation of the map $\pi$.

\bigskip

We will consider $L^{0}$, with the usual operations among random variables,
as a partially ordered ring and we will always assume in the sequel that $%
\tau _{0}$ is a topology on $L^{0}$ such that $(L^{0},\tau _{0})$ is a
topological ring. We do not require that $\tau _{0}$ is a linear topology on 
$L^{0}$ (so that $(L^{0},\tau _{0})$ may not be a topological vector space)
nor that $\tau _{0}$ is locally convex.

\begin{definition}[Topological $L^{0}$-module]
We say that $(E,\tau )$ is a topological $L^{0}$-module if $E$ is a $L^{0}$%
-module and $\tau $ is a topology on $E$ such that the module operation

(i) $(E,\tau)\times (E,\tau)\rightarrow (E,\tau)$, $(x_1,x_2) \mapsto
x_1+x_2 $,

(ii) $(L^0,\tau_0)\times (E,\tau)\rightarrow (E,\tau)$, $(\gamma,x_2)
\mapsto \gamma x_2$

\noindent are continuous w.r.t. the corresponding product topology.
\end{definition}

\begin{definition}[Duality for $L^{0}$-modules]
\label{duality} For a topological $L^{0}$-module $(E,\tau )$, we denote 
\begin{equation}
E^{\ast }:=\{x^{\ast }:(E,\tau )\rightarrow (L^{0},\tau _{0})\mid x^{\ast }%
\text{ is a continuous module homomorphism}\}.  \label{hom1}
\end{equation}%
It is easy to check that $(E,E^{\ast },\langle \cdot ,\cdot \rangle )$ is a
dual pair, where the pairing is given by $\langle x,x^{\ast }\rangle
=x^{\ast }(x).$ Every $x^{\ast }\in E^{\ast }$ is $L^{0}$-linear in the
following sense: for all $\,\alpha ,\beta \in L^{0}$ and $x_{1},x_{2}\in E$ 
\begin{equation*}
x^{\ast }(\alpha x_{1}+\beta x_{2})=\alpha x^{\ast }(x_{1})+\beta x^{\ast
}(x_{2}). 
\end{equation*}%
In particular, $x^{\ast }(x_{1}\mathbf{1}_{A}+x_{2}\mathbf{1}%
_{A^{C}})=x^{\ast }(x_{1})\mathbf{1}_{A}+x^{\ast }(x_{2})\mathbf{1}_{A^{C}}$.
\end{definition}

\begin{definition}
A map $\|\cdot\|:E\rightarrow L^0_+$ is a $L^0$-seminorm on $E$ if

\begin{enumerate}
\item[(i)] $\|\gamma x\|=|\gamma|\|x\|$ for all $\gamma\in L^0$ and $x\in E$,

\item[(ii)] $\|x_1+x_2\|\leq \|x_1\|+\|x_2\|$ for all $x_1,x_2\in E$.

The $L^{0}$-seminorm $\Vert \cdot \Vert $ becomes a $L^{0}$-norm if in
addition

\item[(iii)] $\Vert x\Vert =0$ implies $x=0$.
\end{enumerate}
\end{definition}

We will consider families of $L^{0}$-seminorms $\mathcal{Z}$ satisfying in
addition the property:%
\begin{equation}
\sup \{\Vert x\Vert \mid \Vert x\Vert \in \mathcal{Z}\}=0\text{ iff }x=0,
\label{mi}
\end{equation}%
As clearly pointed out in \cite{Guo}, one family $\mathcal{Z}$ of $L^{0}$%
-seminorms on $E$ may induce on $E$ more than one topology $\tau $ such that 
$\left\{ x_{\alpha }\right\} $ converges to $x$ in $(E,\tau )$ iff $\Vert
x_{\alpha }-x\Vert $ converges to $0$ in $(L^{0},\tau _{0})$ for each $\Vert
\cdot \Vert \in \mathcal{Z}$. Indeed, also the topology $\tau _{0}$\ on $%
L^{0}$ play a role in the convergence.

\begin{definition}[$L^{0}$-module associated to $\mathcal{Z}$]
\label{defez}We say that $(E,\mathcal{Z},\tau )$ is a $L^{0}$-module
associated to $\mathcal{Z}$ if:

\begin{enumerate}
\item $\mathcal{Z}$ is a family of $L^{0}$-seminorms satisfying (\ref{mi}),

\item $(E,\tau )$ is a topological $L^{0}$-module,

\item A net $\left\{ x_{\alpha }\right\} $ converges to $x$ in $(E,\tau )$
iff $\Vert x_{\alpha }-x\Vert $ converges to $0$ in $(L^{0},\tau _{0})$ for
each $\Vert \cdot \Vert \in \mathcal{Z}.$
\end{enumerate}
\end{definition}

Remark 2.2 in \cite{Guo} shows that any random locally convex module over $%
\mathbb{R}$ with base $(\Omega ,\mathcal{G},\mathbb{P}),$ according to
Definition 2.1 \cite{Guo}, is a $L^{0}$\textbf{-}module $(E,\mathcal{Z},\tau
)$ associated to a family $\mathcal{Z}$ of $L^{0}$-seminorms, according to
the previous definition.

Proposition \ref{LLSC} holds if the topological structure of $(E,\mathcal{Z}%
,\tau )$ allows for appropriate separation theorems. We now introduce two
assumptions that are tailor made for the statements in Proposition \ref{LLSC}%
, but in the following subsection we provide interesting and general
examples of $L^{0}$\textbf{-}module associated to $\mathcal{Z}$ that fulfill
these assumptions.

\begin{description}
\item[Separation Assumptions] Let $E$ be a topological $L^{0}$\textbf{-}%
module, let $E^{\ast }$ be defined in (\ref{hom1}) and let $\mathcal{C}%
_{0}\subseteq E$ be nonempty, $L^{0}$-convex and satisfy (CSet).

\item[S-Open] If $\mathcal{C}_{0}$ is also open and $\{x\}\mathbf{1}_{A}\cap 
\mathcal{C}_{0}\mathbf{1}_{A}=\emptyset $ for every $A\in \mathcal{G}$ s.t. $%
P(A)>0$, then there exists $x^{\ast }\in E^{\ast }$ s.t. $x^{\ast
}(x)>x^{\ast }(\xi )\quad \forall \,\xi \in \mathcal{C}_{0}.$

\item[S-Closed] If $\mathcal{C}_{0}$ is also closed and $\{x\}\mathbf{1}%
_{A}\cap \mathcal{C}_{0}\mathbf{1}_{A}=\emptyset $ for every $A\in \mathcal{G%
}$ s.t. $P(A)>0$, then there exists $x^{\ast }\in E^{\ast }$ s.t. $x^{\ast
}(x)>x^{\ast }(\xi )\quad \forall \,\xi \in \mathcal{C}_{0}.$
\end{description}

\begin{lemma}
\label{paste}.

\begin{enumerate}
\item Let $E$ be a topological $L^{0}$\textbf{-}module. If $\mathcal{C}%
_{i}\subseteq E$, $i=1,2,$ are open and non empty and $A\in \mathcal{G}$,
then the set $\mathcal{C}_{1}\mathbf{1}_{A}+\mathcal{C}_{2}\mathbf{1}%
_{A^{C}} $ is open.

\item Let $(E,\mathcal{Z},\tau )$ be $L^{0}$\textbf{-}module associated to $%
\mathcal{Z}$. Then for any net $\{\xi _{\alpha }\}\subseteq E$, $\xi \in E$, 
$\eta \in E$ and $A\in \mathcal{G}$ 
\begin{equation*}
\xi _{\alpha }\overset{\tau }{\rightarrow }\xi \Longrightarrow (\xi _{\alpha
}1_{A}+\eta 1_{A^{C}})\overset{\tau }{\rightarrow }(\xi 1_{A}+\eta
1_{A^{C}}). 
\end{equation*}
\end{enumerate}
\end{lemma}

\begin{proof}
1. To show this claim let $x:=x_{1}\mathbf{1}_{A}+x_{2}\mathbf{1}_{A^{C}}$
with $x_{i}\in \mathcal{C}_{i}$ and let $U_{0}$ be a neighborhood of $0$
satisfying $x_{i}+U_{0}\subseteq \mathcal{C}_{i}$. Then the set $%
U:=(x_{1}+U_{0})\mathbf{1}_{A}+(x_{2}+U_{0})\mathbf{1}_{A^{C}}=x+U_{0}%
\mathbf{1}_{A}+U_{0}\mathbf{1}_{A^{C}}$ is contained in $\mathcal{C}_{1}%
\mathbf{1}_{A}+\mathcal{C}_{2}\mathbf{1}_{A^{C}}$ and it is a neighborhood
of $x$, since $U_{0}\mathbf{1}_{A}+U_{0}\mathbf{1}_{A^{C}}$ contains $U_{0}$
and is therefore a neighborhood of $0$.

2. Observe that a seminorm satisfies $\Vert 1_{A}(\xi _{\alpha }-\xi )\Vert
=1_{A}\Vert \xi _{\alpha }-\xi \Vert \leq \Vert \xi _{\alpha }-\xi \Vert $
and therefore, by condition 3. in Definition \ref{defez} the claim follows.
In particular, $\xi _{\alpha }\overset{\tau }{\rightarrow }\xi
\Longrightarrow (\xi _{\alpha }1_{A})\overset{\tau }{\rightarrow }(\xi
1_{A}) $.
\end{proof}

\begin{proposition}
\label{LLSC}Let $(E,\mathcal{Z},\tau )$ be $L^{0}$\textbf{-}module
associated to $\mathcal{Z}$ and suppose that $\mathcal{C}\subseteq E$
satisfies (CSet).

\begin{enumerate}
\item Suppose that the strictly positive cone $L_{++}^{0}$ is $\tau _{0}$%
-open and that there exist $x_{0}^{\prime }\in E^{\ast }$ and $x_{0}\in E$
such that $x_{0}^{\prime }(x_{0})>0.$ Under Assumption S-Open, if $\mathcal{C%
}$ is open and $L^{0}$-convex then $\mathcal{C}$ is conditionally evenly
convex.

\item Under Assumption S-Closed, if $\mathcal{C}$ is closed and $L^{0}$%
-convex then it is conditionally evenly convex.
\end{enumerate}
\end{proposition}

\begin{proof}
1. Let $\mathcal{C}\subseteq E$ be open, $L^{0}$-convex, $\mathcal{C}\neq
\varnothing $ and let $A_{\mathcal{C}}\in \mathcal{G}$ be the maximal set
given in the Notation \ref{trivialcomponent}, being $D_{\mathcal{C}}$ its
complement. Suppose that $x$ is outside $\mathcal{C}$, i.e. $x\in E$
satisfies $\{x\}\mathbf{1}_{A}\cap \mathcal{C}\mathbf{1}_{A}=\varnothing $
for every $A\in \mathcal{G}$, $A\subseteq D_{\mathcal{C}}$, $P(A)>0$. Define
the $L^{0}$-convex set 
\begin{equation*}
\mathcal{E}:=\{\xi \in E\mid x_{0}^{\prime }(\xi )>x_{0}^{\prime
}(x)\}=(x_{0}^{\prime })^{-1}(x_{0}^{\prime }(x)+L_{++}^{0})
\end{equation*}
and notice that $\{x\}\mathbf{1}_{A}\cap \mathcal{E}\mathbf{1}_{A}=\emptyset 
$ for every $A\in \mathcal{G}$. As $L_{++}^{0}$ is $\tau _{0}$-open, $%
\mathcal{E}$ is open in $E.$ As $x_{0}^{\prime }(x_{0})>0,$ then $%
(x+x_{0})\in \mathcal{E}$ and $\mathcal{E}$ is non-empty.

Then the set $\mathcal{C}_{0}=\mathcal{C}\mathbf{1}_{D_{\mathcal{C}}}+%
\mathcal{E}\mathbf{1}_{A_{\mathcal{C}}}$ is $L^{0}$-convex, open (by Lemma %
\ref{paste}) and satisfies $\{x\}\mathbf{1}_{A}\cap \mathcal{C}_{0}\mathbf{1}%
_{A}=\emptyset $ for every $A\in \mathcal{G}$ s.t. $P(A)>0$. Assumption
S-Open guarantees the existence of $x^{\ast }\in E^{\ast }$ s.t. $x^{\ast
}(x)>x^{\ast }(\xi )\quad \forall \,\xi \in \mathcal{C}_{0},$ which implies $%
x^{\ast }(x)>x^{\ast }(\xi )\quad $on $D_{\mathcal{C}}$, $\forall \,\xi \in 
\mathcal{C}$. Hence, by Theorem \ref{main}, $\mathcal{C}$ is conditionally
evenly convex.

2. Let $\mathcal{C}\subset E$ be closed, $L^{0}$-convex, $\mathcal{C}\neq
\varnothing $ and suppose that $x\in E$ satisfies $\{x\}\mathbf{1}_{A}\cap 
\mathcal{C}\mathbf{1}_{A}=\varnothing $ for every $A\in \mathcal{G}$, $%
A\subseteq D_{\mathcal{C}}$, $\mathbb{P}(A)>0$. Let $\mathcal{C}_{0}=%
\mathcal{C}\mathbf{1}_{D_{\mathcal{C}}}+\{x+\varepsilon \}\mathbf{1}_{A_{%
\mathcal{C}}}$ where $\varepsilon \in L_{++}^{0}$. Clearly $\mathcal{C}_{0}$
is $L^{0}$-convex. In order to prove that $\mathcal{C}_{0}$ is closed
consider any net $\xi _{\alpha }\overset{\tau }{\rightarrow }\xi $, $\{\xi
_{\alpha }\}\subset \mathcal{C}_{0}$. Then $\xi _{\alpha }=Z_{\alpha }%
\mathbf{1}_{D_{\mathcal{C}}}+\{x+\varepsilon \}\mathbf{1}_{A_{\mathcal{C}}}$%
, with $Z_{\alpha }\in \mathcal{C}$, and $(x+\varepsilon )\mathbf{1}_{A_{%
\mathcal{C}}}=\xi 1_{A_{\mathcal{C}}}.$ Take any $\eta \in \mathcal{C}$. As $%
\mathcal{C}$ is $L^{0}$-convex, $\xi _{\alpha }\mathbf{1}_{D_{\mathcal{C}%
}}+\eta \mathbf{1}_{A_{\mathcal{C}}}=Z_{\alpha }\mathbf{1}_{D_{\mathcal{C}%
}}+\eta \mathbf{1}_{A_{\mathcal{C}}}\in \mathcal{C}$ and, by Lemma \ref%
{paste}, $\xi _{\alpha }\mathbf{1}_{D_{\mathcal{C}}}+\eta \mathbf{1}_{A_{%
\mathcal{C}}}\overset{\tau }{\rightarrow }\xi \mathbf{1}_{D_{\mathcal{C}%
}}+\eta \mathbf{1}_{A_{\mathcal{C}}}:=Z\in \mathcal{C}$, as $\mathcal{C}$ is
closed. Therefore, $\xi =Z\mathbf{1}_{D_{\mathcal{C}}}+\{x+\varepsilon \}%
\mathbf{1}_{A_{\mathcal{C}}}\in \mathcal{C}_{0}$. Since $\mathcal{C}_{0}$ is
closed, $L^{0}$-convex and $\{x\}\mathbf{1}_{A}\cap \mathcal{C}_{0}\mathbf{1}%
_{A}=\emptyset $ for every $A\in \mathcal{G}$, assumption S-Closed
guarantees the existence of $x^{\ast }\in E^{\ast }$ s.t. $x^{\ast
}(x)>x^{\ast }(\xi )\quad \forall \,\xi \in \mathcal{C}_{0},$ which implies $%
x^{\ast }(x)>x^{\ast }(\xi )\quad $on $D_{\mathcal{C}}$, $\forall \,\xi \in 
\mathcal{C}$. Hence, by Theorem \ref{main}, $\mathcal{C}$ is conditionally
evenly convex.
\end{proof}

\begin{proposition}
Let $(E,\mathcal{Z},\tau )$ and $E^{\ast }$ be respectively as in
definitions \ref{duality} and \ref{defez}, and let $\tau _{0}$ be a topology
on $L^{0}$ such that the positive cone $L_{+}^{0}$ is closed. Then any
conditionally evenly convex $L^{0}$-cone containing the origin is closed.
\end{proposition}

\begin{proof}
From (\ref{222}) and the bipolar Theorem \ref{bipolar} we know that 
\begin{equation*}
\mathcal{C}=\mathcal{C}^{\circ \circ }\,=\bigcap_{x^{\prime }\in \mathcal{C}%
^{\circ }}\left\{ x\in E\mid \langle x,x^{\prime }\rangle \leq 0\text{ on }%
D_{\mathcal{C}}\right\} .
\end{equation*}%
We only need to prove that $\mathcal{S}_{x^{\prime }}=\left\{ x\in E\mid
\langle x,x^{\prime }\rangle \leq 0\text{ on }D_{\mathcal{C}}\right\} $ is
closed for any $x^{\prime }\in \mathcal{C}^{\circ }$. Let $x_{\alpha }\in 
\mathcal{S}_{x^{\prime }}$ be a net such that $x_{\alpha }\overset{\tau }{%
\rightarrow }x$. Since $x^{\prime }\in E^{\ast }$ is continuous we have $%
Y_{\alpha }=:\langle x_{\alpha },x^{\prime }\rangle \overset{\tau _{0}}{%
\rightarrow }Y=:\langle x,x^{\prime }\rangle $, with $Y_{\alpha }\leq 0$ on $%
D_{\mathcal{C}}$. We surely have that $x_{\alpha }\mathbf{1}_{D_{\mathcal{C}%
}}\overset{\tau }{\rightarrow }x\mathbf{1}_{D_{\mathcal{C}}}$ which implies
that $Y_{\alpha }\mathbf{1}_{D_{\mathcal{C}}}\overset{\tau _{0}}{\rightarrow 
}Y\mathbf{1}_{D_{\mathcal{C}}}$. Since $-Y_{\alpha }\mathbf{1}_{D_{\mathcal{C%
}}}\in L_{+}^{0}$ for every $\alpha $ and $L_{+}^{0}$ is closed we conclude
that $Y=\langle x,x^{\prime }\rangle \leq 0$ on $D_{\mathcal{C}}$.
\end{proof}

\subsection{On $L^{0}$\textbf{-}module associated to $\mathcal{Z}$
satisfying S-Open and S-Closed}

Based on the results of Guo \cite{Guo} and Filipovic et al.\cite{FKV09}, we
show that a family of seminorms on $E$ may induce more than one topology on
the $L^{0}$-module $E$ and that these topologies satisfy the assumptions
S-Open and S-Closed.

These examples are quite general and therefore supports the claim made in
the previous section about the relevance of conditional evenly convex sets.
A concrete and significant example, already introduced in Section 2, is
provided next. To help the reader in finding further details we use the same
notations and definitions given in \cite{FKV09} and \cite{Guo}.

\begin{example}[\protect\cite{FKV10}]
\label{exmod}Let $\mathcal{F}$ be a sigma algebra containing in $\mathcal{G}$
and consider the generalized conditional expectation of $\mathcal{F}$%
-measurable non negative random variables: $E[\cdot |\mathcal{G}%
]:L_{+}^{0}(\Omega ,\mathcal{F},\mathbb{P})\rightarrow \bar{L}_{+}^{0}:=\bar{%
L}_{+}^{0}(\Omega ,\mathcal{G},\mathbb{P})$ 
\begin{equation*}
E[x|\mathcal{G}]=:\lim_{n\rightarrow +\infty }E[x\wedge n|\mathcal{G}]. 
\end{equation*}%
Let $p\in \lbrack 1,\infty ]$ and consider the $L^{0}$-module defined as 
\begin{equation*}
L_{\mathcal{G}}^{p}(\mathcal{F})=:\{x\in L^{0}(\Omega ,\mathcal{F},\mathbb{P}%
)\mid \Vert x|\mathcal{G}\Vert _{p}\in L^{0}(\Omega ,\mathcal{G},\mathbb{P}%
)\} 
\end{equation*}%
where $\Vert \cdot |\mathcal{G}\Vert _{p}$ is the $L^{0}$-norm assigned by 
\begin{equation}
\Vert x|\mathcal{G}\Vert _{p}=:\left\{ 
\begin{array}{cc}
E[|x|^{p}|\mathcal{G}]^{\frac{1}{p}} & \text{ if }p<+\infty \\ 
\inf \{y\in \bar{L}^{0}(\mathcal{G})\mid y\geq |x|\} & \text{ if }p=+\infty%
\end{array}%
\right.  \label{norm}
\end{equation}%
Then $L_{\mathcal{G}}^{p}(\mathcal{F})$ becomes a $L^{0}$-normed module
associated to the norm $\Vert \cdot |\mathcal{G}\Vert _{p}$ having the
product structure: 
\begin{equation*}
L_{\mathcal{G}}^{p}(\mathcal{F})=L^{0}(\mathcal{G})L^{p}(\mathcal{F}%
)=\{yx\mid y\in L^{0}(\mathcal{G}),\;x\in L^{p}(\mathcal{F})\}. 
\end{equation*}%
For $p<\infty $, any $L^{0}$-linear continuous functional $\mu :L_{\mathcal{G%
}}^{p}(\mathcal{F})\rightarrow L^{0}$ can be identified with a random
variable $z\in L_{\mathcal{G}}^{q}(\mathcal{F})$ as $\mu (\cdot )=E[z\cdot |%
\mathcal{G}]$ where $\frac{1}{p}+\frac{1}{q}=1$. So we can identify $E^{\ast
}$ with $L_{\mathcal{G}}^{q}(\mathcal{F})$.
\end{example}

The two different topologies on $E$ depend on which topology is selected on $%
L^{0}$: either the uniform topology or the topology of convergence in
probability. \newline
The two topologies on $E$ will collapse to the same one whenever $\mathcal{G}%
=\sigma (\varnothing )$ is the trivial sigma algebra, but in general present
different structural properties.

We set:%
\begin{equation*}
\Vert x\Vert _{\mathcal{S}}:=\sup \{\Vert x\Vert \mid \Vert x\Vert \in 
\mathcal{S}\} 
\end{equation*}%
for any finite subfamily $\mathcal{S}\subset \mathcal{Z}$ \ of $L^{0}$%
-seminorms. Recall from the assumption given in equation (\ref{mi}) that $%
\Vert x\Vert _{\mathcal{S}}=0$ if and only if $x=0$.

\paragraph{The uniform topology $\protect\tau _{c}$ \protect\cite{FKV09}.}

In this case, $L^{0}$ is equipped with the following uniform topology. For
every $\varepsilon \in L_{++}^{0},$ the ball $B_{\varepsilon }:=\{Y\in
L^{0}\mid |Y|\leq \varepsilon \}$ centered in $0\in L^{0}$ gives the
neighborhood basis of $0$. A set $V\subset L^{0}$ is a neighborhood of $Y\in
L^{0}$ if there exists $\varepsilon \in L_{++}^{0}$ such that $%
Y+B_{\varepsilon }\subset V$. A set $V$ is open if it is a neighborhood of
all $Y\in V$. A net converges in this topology, namely $Y_{N}\overset{|\cdot
|}{\rightarrow }Y$ if for every $\varepsilon \in L_{++}^{0}$ there exists $%
\overline{N}$ such that $|Y-Y_{N}|<\varepsilon $ for every $N>\overline{N}$.
In this case the space $(L^{0},|\cdot |)$ looses the property of being a
topological vector space. In this topology the positive cone $L_{+}^{0}$ is
closed and the strictly positive cone $L_{++}^{0}$ is open. 

Under the assumptions that there exists an $x\in E$ such that $x\mathbf{1}%
_{A}\neq 0$ for every $A\in \mathcal{G}$ and that the topology $\tau $ on $E$
is Hausdorff, Theorem 2.8 in \cite{FKV09} guarantees the existence of $%
x_{0}\in E$ and $x_{0}^{\prime }\in E^{\ast }$ such that $x_{0}^{\prime
}(x_{0})>0$. This and the next item 2 allow the application of Proposition  %
\ref{LLSC}.

\bigskip

A family $\mathcal{Z}$ of $L^{0}$-seminorms on $E$ induces a topology on $E$
in the following way. For any finite $\mathcal{S}\subset \mathcal{Z}$ and $%
\varepsilon \in L_{++}^{0}$ define 
\begin{eqnarray*}
&&U_{\mathcal{S},\varepsilon }:=\left\{ x\in E\mid \Vert x\Vert _{\mathcal{S}%
}\leq \varepsilon \right\} \\
&&\mathcal{U}:=\{U_{\mathcal{S},\varepsilon }\mid \mathcal{S}\subset 
\mathcal{Z}\text{ finite and }\varepsilon \in L_{++}^{0}\}.
\end{eqnarray*}%
$\mathcal{U}$ gives a convex neighborhood base of $0$ and it induces a
topology on $E$ denoted by $\tau _{c}$. We have the following properties:

\begin{enumerate}
\item $(E,\mathcal{Z},\tau _{c})$ is a $(L^{0},|\cdot |)$-module associated
to $\mathcal{Z}$, which is also a locally convex topological $L^{0}$-module
(see Proposition 2.7 \cite{Guo}),

\item $(E,\mathcal{Z},\tau _{c})$ satisfies S-Open and S-Closed (see
Theorems 2.6 and 2.8 \cite{FKV09}),

\item Any topological $(L^{0},|\cdot |)$ module $(E,\tau )$ is locally
convex if and only if $\tau $ is induced by a family of $L^{0}$-seminorms,
i.e. $\tau \equiv \tau _{c}$, (see Theorem 2.4 \cite{FKV09}).
\end{enumerate}

\paragraph{A probabilistic topology $\protect\tau _{\protect\epsilon ,%
\protect\lambda }$ \protect\cite{Guo}}

The second topology on the $L^{0}$-module $E$ is a topology of a more
probabilistic nature and originated in the theory of probabilistic metric
spaces (see \cite{SS83}).

Here $L^{0}$ is endowed with the topology $\tau _{\epsilon ,\lambda }$ of
convergence in probability and so the positive cone $L_{+}^{0}$ is $\tau _{0}
$-closed. According to \cite{Guo}, for every $\epsilon ,\lambda \in \mathbb{R%
}$ and a finite subfamily $\mathcal{S}\subset \mathcal{Z}$ of $L^{0}$%
-seminorms we let 
\begin{eqnarray*}
&&\mathcal{V}_{\mathcal{S},\epsilon ,\lambda }:=\{x\in E\mid \mathbb{P}%
(\Vert x\Vert _{\mathcal{S}}<\epsilon )>1-\lambda \} \\
&&\mathcal{V}:=\{\mathcal{U}_{\mathcal{S},\epsilon ,\lambda }\mid \mathcal{S}%
\subset \mathcal{Z}\text{ finite, }\epsilon >0,\,0<\lambda <1\}.
\end{eqnarray*}%
$\mathcal{V}$ gives a neighborhood base of $0$ and it induces a linear
topology on $E$, also denoted by $\tau _{\epsilon ,\lambda }$ (indeed if $%
E=L^{0}$ then this is exactly the topology of convergence in probability).
This topology may not be locally convex, but has the following properties:

\begin{enumerate}
\item $(E,\mathcal{Z},\tau _{\epsilon ,\lambda })$ becomes a $(L^{0},\tau
_{\epsilon ,\lambda })$-module associated to $\mathcal{Z}$ (see Proposition
2.6 \cite{Guo}),

\item $(E,\mathcal{Z},\tau _{\epsilon ,\lambda })$ satisfies S-Closed (see
Theorems 3.6 and 3.9 \cite{Guo}).
\end{enumerate}

Therefore Proposition \ref{LLSC} can be applied.

\section{On Conditionally Evenly Quasi-Convex maps on $L^{0}$-module}

\label{module&map}

As an immediate consequence of Proposition \ref{LLSC} we have that lower
(resp. upper) semicontinuity and quasiconvexity imply evenly quasiconvexity
of $\rho$. From Theorem \ref{EvQco} we then deduce the representation for
lower (resp. upper) semicontinuous quasiconvex maps.

\begin{description}
\item[(LSC)] A map $\pi :E\rightarrow \bar{L}^{0}(\mathcal{G})$ is lower
semicontinuous if for every $Y\in L^0$ the lower level sets $U_Y=\{\xi\in E
| \pi(\xi)\mathbf{1}_{T_{\pi}}\leq Y\}$ are $\tau$-closed.
\end{description}

\begin{corollary}
\label{LSC} Let $(E,\mathcal{Z},\tau )$ and $E^{\prime}=E^{\ast }$ be
respectively as in definitions \ref{duality} and \ref{defez}, satisfying
S-Closed. \newline
If $\pi :E\rightarrow \bar{L}^{0}(\mathcal{G})$ is (REG), (QCO) and (LSC)
then (\ref{rapprPiSup}) holds true.
\end{corollary}

In the upper semicontinuous case we can say more (the proof is postponed to
Section \ref{proof}).

\begin{description}
\item[(USC)] A map $\pi :E\rightarrow \bar{L}^{0}(\mathcal{G})$ is upper
semicontinuous if for every $Y\in L^0$ the lower level sets $U_Y=\{\xi\in E
| \pi(\xi)\mathbf{1}_{T_{\pi}}< Y\}$ are $\tau$-open.
\end{description}

\begin{proposition}
\label{USC} Let $(E,\mathcal{Z},\tau )$ and $E^{\prime}=E^{\ast }$ be
respectively as in Proposition \ref{LLSC} statement 1, satisfying S-Open. 
\newline
If $\pi :E\rightarrow \bar{L}^{0}(\mathcal{G})$ is (REG), (QCO) and (USC)
then 
\begin{equation}
\pi (x)=\max_{x^{\ast} \in E^{\ast}}\mathcal{R}(\langle x,x^{\ast}
\rangle,x^{\ast} ).  \label{rapprPiMax}
\end{equation}
\end{proposition}

In Theorem \ref{EvQco}, $\pi $ can be represented as a supremum but not as a
maximum. The following corollary shows that nevertheless we can find a $%
\mathcal{R}(\langle x,x^{\ast} \rangle,x^{\ast} )$ arbitrary close to $\pi
(x)$.

\begin{corollary}
\label{vicino} Under the same assumption of Theorem \ref{EvQco} or Corollary %
\ref{LSC}, for every $\varepsilon \in L^0_{++}$ there exists $%
x^*_{\varepsilon }\in E^{\ast}$ such that 
\begin{equation}
\pi (x)-\mathcal{R}(\langle x,x^*_{\varepsilon}
\rangle,x^*_{\varepsilon})<\varepsilon \text{ on the set }\{\pi (x)<+\infty
\}.  \label{appr}
\end{equation}
\end{corollary}

\begin{proof}
The statement is a direct consequence of the inequalities (\ref{111})
through (\ref{sequence}) of Step 3 in the proof of Theorem \ref{EvQco}.
\end{proof}

\section{Proofs}

\label{proof}

\begin{notation}
The condition $\mathbf{1}_{A}\left\{ \eta \right\} \cap \mathbf{1}_{A}%
\mathcal{C}\neq \varnothing $ is equivalent to: $\exists \xi \in \mathcal{C}$
s.t. $\mathbf{1}_{A}\eta =\mathbf{1}_{A}\xi .$

\noindent For $\eta \in E,$ $B\in \mathcal{G}$ and $\mathcal{C}\subseteq E$
we say that 
\begin{equation*}
\eta \text{ is outside}\mid _{B} \mathcal{C} \text{ \ if }\forall A\subseteq
B,\text{ }A\in \mathcal{G},\text{ }\mathbb{P}(A)>0\text{, }\mathbf{1}%
_{A}\left\{ \eta \right\} \cap \mathbf{1}_{A}\mathcal{C}=\varnothing . 
\end{equation*}%
If $\mathbb{P}(B)=0$ then $\eta $ is outside$\mid _{B} \mathcal{C}$ is
equivalent to $\eta \in \mathcal{C}.$ Recall that $A_{\mathcal{C}}$ is the
maximal set of $\mathcal{A}(\mathcal{C})=\left\{ B\in \mathcal{G}\mid 
\mathbf{1}_{A}E=\mathbf{1}_{A}\mathcal{C}\right\} $, $D_{\mathcal{C}}$ is
the complement of $A_{\mathcal{C}}$ and that $\eta $ is outside $\mathcal{C}$
if $\eta $ is outside$\mid _{D_{\mathcal{C}}} \mathcal{C}$.
\end{notation}

\begin{remark}
\label{remMAX}By Lemma 2.9 in \cite{FKV09}, we know that any non-empty class 
$\mathcal{A}$ of subsets of a sigma algebra $\mathcal{G}$ has a supremum $%
\emph{ess}.\sup \{\mathcal{A}\}\in \mathcal{G}$ and that if $\mathcal{A}$ is
closed with respect to finite union (i.e. $A_{1},A_{2}\in \mathcal{A}%
\Rightarrow A_{1}\cup A_{2}\in \mathcal{A}$) then there is a sequence $%
A_{n}\in \mathcal{A}$ such that $\emph{ess}.\sup \{\mathcal{A}%
\}=\bigcup\limits_{n\in \mathbb{N}}A_{n}$. Obviously, if $\mathcal{A}$ is
closed with respect to countable union then $\emph{ess}.\sup \{\mathcal{A}%
\}=\bigcup\limits_{n\in \mathbb{N}}A_{n}:=A_{M}\in \mathcal{A}$ is the
maximal element in $\mathcal{A}$.
\end{remark}

For our proofs we need a simplified version of a result proved by Guo
(Theorem 3.13, \cite{Guo}) concerning hereditarily disjoint stratification
of two subsets. We reformulate his result in the following

\begin{lemma}
\label{LemmaGuoSet}Suppose that $\mathcal{C}\subset E$ satisfies $\mathbf{1}%
_{A}\mathcal{C}+\mathbf{1}_{A^{C}}\mathcal{C}\subseteq \mathcal{C}$, for
every $A\in \mathcal{G}$. If there exists $x\in E$ with $x\notin \mathcal{C}$
then there exists a set $H:=H_{\mathcal{C},x}\in \mathcal{G}$ such that $%
\mathbb{P}(H)>0$ and%
\begin{eqnarray}
&&\mathbf{1}_{\Omega \backslash H}\left\{ x\right\} \cap \mathbf{1}_{\Omega
\backslash H}\mathcal{C}\neq \varnothing  \label{11} \\
&&x\text{ is outside}\mid _{H}\mathcal{C}  \label{12}
\end{eqnarray}
\end{lemma}

\noindent The two above conditions guarantee that $H_{\mathcal{C},x}$ is the
largest set $D\in \mathcal{G}$ such that $x$ is outside$\mid _{D}\mathcal{C}$%
.

\begin{lemma}
\label{ln}Suppose that $\mathcal{C}$ satisfies (CSet).

\begin{enumerate}
\item If $x\notin \mathcal{C}$ then the set $H_{\mathcal{C},x}$ defined in
Lemma \ref{LemmaGuoSet} satisfies $H_{\mathcal{C},x}\subseteq D_{\mathcal{C}%
} $ and so $\mathbb{P}(D_{\mathcal{C}})\geq \mathbb{P}(H_{\mathcal{C},x})>0$.

\item If $x$ is outside $\mathcal{C}$ then $\mathbb{P}(H_{\mathcal{C},x})>0$
and $H_{\mathcal{C},x}=D_{\mathcal{C}}$.

\item If $x\notin \mathcal{C}$ then 
\begin{equation}
\chi :=\{y\in E\mid y\text{ is outside }\mathcal{C}\}\neq \varnothing \text{.%
}  \label{111}
\end{equation}
\end{enumerate}
\end{lemma}

\begin{proof}[Proof]
1. Lemma \ref{LemmaGuoSet} shows that $\mathbb{P}(H_{\mathcal{C},x})>0$.
Since $\mathbf{1}_{A_{\mathcal{C}}}E=\mathbf{1}_{A_{\mathcal{C}}}\mathcal{C}$%
, if $x\notin \mathcal{C}$ we necessarily have: $\mathbb{P}(H_{\mathcal{C}%
,x}\cap A_{\mathcal{C}})=0$ and therefore $H_{\mathcal{C},x}\subseteq D_{%
\mathcal{C}}$.

2. If $x$ is outside$\mid \mathcal{C}$ then $x$ is outside$\mid _{D_{%
\mathcal{C}}}\mathcal{C}$ and $x\notin \mathcal{C}$. The thesis follows from 
$H_{\mathcal{C},x}\subseteq D_{\mathcal{C}}$ and the fact that $H_{\mathcal{C%
},x}$ is the largest set $D\in \mathcal{G}$ for which $x$ is outside$\mid
_{D}\mathcal{C}$.

3. is a consequence of Lemma \ref{external} (see Appendix) item 1.
\end{proof}

\bigskip

\begin{proof}[Proof of Theorem \protect\ref{main}]
(1)$\Rightarrow $(2). Let $\mathcal{L}\subset E^{\prime }$, $Y_{x^{\prime
}}\in L^{0}$ and let 
\begin{equation*}
\mathcal{C}=:\bigcap\limits_{x^{\prime }\in \mathcal{L}}\left\{ \xi \in
E\mid \langle \xi ,x^{\prime }\rangle <Y_{x^{\prime }}\text{ on }D_{\mathcal{%
C}}\right\} ,
\end{equation*}%
which clearly satisfies $\mathcal{C}^{cc}=\mathcal{C}$. By definition, if
there exists $x\in E$ s.t. $x$ is outside $\mathcal{C}$ then $\mathbf{1}%
_{A}\left\{ x\right\} \cap \mathbf{1}_{A}\mathcal{C}=\varnothing $ $\forall
A\subseteq D_{\mathcal{C}},$ $A\in \mathcal{G}$, $\mathbb{P}(A)>0$, and
therefore by the definition of $\mathcal{C}$ there exists $\,x^{\prime }\in 
\mathcal{L}$ s.t. $\langle x,x^{\prime }\rangle \geq Y_{x^{\prime }}$ on $D_{%
\mathcal{C}}.$ Hence: $\langle x,x^{\prime }\rangle \geq Y_{x^{\prime
}}>\langle \xi ,x^{\prime }\rangle $ on $D_{\mathcal{C}}$ for all $\xi \in 
\mathcal{C}.$

(2)$\Rightarrow $ (1) We are assuming that $\mathcal{C}$ is (CSet), and
there exists $x\in E$ s.t. $x\notin \mathcal{C}$ (otherwise $\mathcal{C}=E$%
). From (\ref{111}) we know that $\chi =\{y\in E\mid y$ is outside $\mathcal{%
C}\}$ is nonempty. By assumption, for all $y\in \chi $ there exists $\xi
_{y}^{\prime }\in E^{\prime }$ such that $\langle \xi ,\xi _{y}^{\prime
}\rangle <\langle y,\xi _{y}^{\prime }\rangle $ on $D_{\mathcal{C}%
},\;\forall \,\xi \in \mathcal{C}$. Define 
\begin{equation*}
B_{y}:=\{\xi \in E\mid \langle \xi ,\xi _{y}^{\prime }\rangle <\langle y,\xi
_{y}^{\prime }\rangle \text{ on }D_{\mathcal{C}}\}.
\end{equation*}%
$B_{y}$ clearly depends also on the selection of the $\xi _{y}^{\prime }\in
E^{\prime }$ associated to $y$ and on $\mathcal{C}$, but this notation will
not cause any ambiguity. We have: $\mathcal{C}\subseteq B_{y}$ for all $y\in
\chi $, and $\mathcal{C}\subseteq \bigcap\limits_{y\in \chi }B_{y}.$ We now
claim that $x\notin \mathcal{C}$ implies $x\notin \bigcap\limits_{y\in \chi
}B_{y}$, thus showing 
\begin{equation}
\mathcal{C}=\bigcap\limits_{y\in \chi }B_{y}=\bigcap\limits_{\xi
_{y}^{\prime }\in \mathcal{L}}\{\xi \in E\mid \langle \xi ,\xi _{y}^{\prime
}\rangle <Y_{\xi _{y}^{\prime }}\text{ on }D_{\mathcal{C}}\},
\label{semipiani}
\end{equation}%
where $\mathcal{L}:=\left\{ \xi _{y}^{\prime }\in E^{\prime }\mid y\in \chi
\right\} $, $Y_{\xi _{y}^{\prime }}:=\langle y,\xi _{y}^{\prime }\rangle \in
L^{0}$, and the thesis is proved.

Suppose that $x\notin \mathcal{C}$, then, by Lemma \ref{LemmaGuoSet}, $x$ is
outside$\mid _{H}$ $\mathcal{C}$, where we set for simplicity $H=H_{\mathcal{%
C},x}$. Take any $y\in \chi \neq \varnothing $ and define $%
y_{0}:=x1_{H}+y1_{\Omega \backslash H}\in \chi $. Take $B_{y_{0}}=\{\xi \in
E\mid \langle \xi ,\xi _{y_{0}}^{\prime }\rangle <\langle y_{0},\xi
_{y_{0}}^{\prime }\rangle $ on $D_{\mathcal{C}}\}$ where $\xi
_{y_{0}}^{\prime }\in E^{\prime }$ is the element associated to $y_{0}$. If $%
x\in B_{y_{0}}$ then we would have: $\langle x,\xi _{y_{0}}^{\prime }\rangle
<\langle y_{0},\xi _{y_{0}}^{\prime }\rangle =\langle x,\xi _{y_{0}}^{\prime
}\rangle $ on $H\subseteq D_{\mathcal{C}}$, by Lemma \ref{ln} item 1, which
is a contradiction, since $\mathbb{P}(H)>0$. Hence $x\notin
B_{y_{0}}\supseteq \bigcap\limits_{y\in \chi }B_{y}.$
\end{proof}

\begin{proposition}
\label{prop24}Under the same assumptions of Theorem \ref{main}, the
following are equivalent:

\begin{enumerate}
\item $\mathcal{C}$ is conditionally evenly convex

\item for every $x\in E$, $x\notin \mathcal{C}$, there exists $x^{\prime
}\in E^{\prime }$ such that 
\begin{equation*}
\langle \xi ,x^{\prime }\rangle <\langle x,x^{\prime }\rangle \text{ on }H_{%
\mathcal{C},x}\;\forall \,\xi \in \mathcal{C}, 
\end{equation*}%
where $H_{\mathcal{C},x}$ is defined in Lemma \ref{LemmaGuoSet}.
\end{enumerate}
\end{proposition}

\begin{proof}
(1)$\Longrightarrow $(2): We know that $\mathcal{C}$ satisfies (CSet). As $%
x\notin \mathcal{C}$, from (\ref{111}) and Lemma \ref{LemmaGuoSet} we know
that there exists $y\in E$ s.t. $y$ is outside $C$ and that $H=:H_{\mathcal{C%
},x}$\ satisfies $\mathbb{P}(H)>0$. Define $\tilde{x}=x\mathbf{1}_{H}+y%
\mathbf{1}_{\Omega \backslash H}$. Then $\tilde{x}$ is outside $\mathcal{C}$
and by Theorem \ref{main} item 2 there exists $x^{\prime }\in E^{\prime }$ 
\begin{equation*}
\langle \xi ,x^{\prime }\rangle <\langle \tilde{x},x^{\prime }\rangle \text{
on }D_{\mathcal{C}},\;\forall \,\xi \in \mathcal{C}.
\end{equation*}%
This implies the thesis since $\langle \tilde{x},x^{\prime }\rangle =\langle
x,x^{\prime }\rangle \mathbf{1}_{H}+\langle y,x^{\prime }\rangle \mathbf{1}%
_{\Omega \backslash H}$ and $H\subseteq D_{\mathcal{C}}$.

(2)$\Longrightarrow $(1): We show that item 2 of Theorem \ref{main} holds
true. This is trivial since if $x$ is outside $\mathcal{C}$ then $x\notin 
\mathcal{C}$ and $H_{\mathcal{C},x}=D_{\mathcal{C}}.$
\end{proof}

\bigskip

\begin{proof}[Proof of Theorem \protect\ref{bipolar}]
Item (1) is straightforward; the fact that $\mathcal{C}^{\circ \circ }$ is
conditionally evenly convex follows from the definition; the proof of $%
\mathcal{C}\subseteq \mathcal{C}^{\circ \circ }$ is also obvious. We now
suppose that $\mathcal{C}$ is conditionally evenly convex and show the
reverse inequality $\mathcal{C}^{\circ \circ }\subseteq \mathcal{C}$. By
contradiction let $x\in \mathcal{C}^{\circ \circ }$ and $x\notin \mathcal{C}$%
. As $\mathcal{C}$ is conditionally evenly convex we apply Proposition \ref%
{prop24} and find $x^{\prime }\in E^{\prime }$ such that 
\begin{equation*}
\langle \xi ,x^{\prime }\rangle <\langle x,x^{\prime }\rangle \text{ on }H_{%
\mathcal{C},x}\text{ for all }\xi \in \mathcal{C}.
\end{equation*}%
Since $0\in \mathcal{C}$, $0=\langle 0,x^{\prime }\rangle <\langle
x,x^{\prime }\rangle $ on $H=:H_{\mathcal{C},x}$. Take any $x_{1}^{\prime
}\in \mathcal{C}^{\circ }$ (which is clearly not empty) and set \ $y^{\prime
}:=\frac{x^{\prime }}{\langle x,x^{\prime }\rangle }\mathbf{1}%
_{H}+x_{1}^{\prime }\mathbf{1}_{\Omega \backslash H}$. Then $y^{\prime }\in
E^{\prime }$ and $\langle \xi ,y^{\prime }\rangle <1$ on $D_{\mathcal{C}}$
for all $\xi \in \mathcal{C}$. This implies $y^{\prime }\in \mathcal{C}%
^{\circ }$. In addition, $\langle x,y^{\prime }\rangle =1$ on $H\subseteq D_{%
\mathcal{C}}$ which is in contradiction with $x\in \mathcal{C}^{\circ \circ
} $.
\end{proof}

\paragraph{General properties of $\mathcal{R}(Y,\protect\mu )$}

Following the path traced in \cite{FM09}, we adapt to the module framework
the proofs of the foremost properties holding for the function $\mathcal{R}%
:L^{0}(\mathcal{G})\times E^{\ast}\rightarrow \bar{L}^{0}(\mathcal{G})$
defined in (\ref{1212}). Let the effective domain of the function $\mathcal{R%
}$ be: 
\begin{equation}
\Sigma _{\mathcal{R}}:=\{(Y,\mu )\in L^{0}(\mathcal{G})\times E^{\ast}\mid
\exists \xi \in E\text{ s.t. }\mu (\xi )\geq Y\}.  \label{domainMod}
\end{equation}

\begin{lemma}
\label{down} Let $\mu \in E^{\ast}$, $X\in E$ and $\pi :E\rightarrow \bar{L}%
^{0}(\mathcal{G})$ satisfy (REG).

\noindent i) $\mathcal{R}(\cdot ,\mu )$ is monotone non decreasing.

\noindent ii) $\mathcal{R}(\Lambda \mu (X),\Lambda \mu )=\mathcal{R}(\mu
(X),\mu )$ for every $\Lambda \in L^{0}(\mathcal{G})$.

\noindent iii) For every $Y\in L^{0}(\mathcal{G})$ and $\mu \in E^{\ast}$,
the set 
\begin{equation*}
\mathcal{A}_{\mu }(Y)\circeq \{\pi (\xi )\,|\,\xi \in E,\;\mu (\xi )\geq Y\} 
\end{equation*}%
is downward directed in the sense that for every $\pi (\xi _{1}),\pi (\xi
_{2})\in \mathcal{A}_{\mu }(Y)$ there exists $\pi (\xi ^{\ast })\in \mathcal{%
A}_{\mu }(Y)$ such that $\pi (\xi ^{\ast })\leq \min \{\pi (\xi _{1}),\pi
(\xi _{2})\}$.

In addition, if $\mathcal{R}(Y,\mu )<\alpha $ for some $\alpha \in L^{0}(%
\mathcal{G})$ then there exists $\xi $ such that $\mu (\xi )\geq Y$ and $\pi
(\xi )<\alpha $.

\noindent iv) For every $A\in \mathcal{G}$, ($Y,\mu )\in \Sigma _{\mathcal{R}%
}$ 
\begin{eqnarray}
\mathcal{R}(Y,\mu )\mathbf{1}_{A} &=&\inf_{\xi \in E}\left\{ \pi (\xi )%
\mathbf{1}_{A}\mid Y\geq \mu (X)\right\}  \label{TakingOutM} \\
&=&\inf_{\xi \in E}\left\{ \pi (\xi )\mathbf{1}_{A}\mid Y\mathbf{1}_{A}\geq
\mu (X\mathbf{1}_{A})\right\} =\mathcal{R}(Y\mathbf{1}_{A},\mu )\mathbf{1}%
_{A}  \label{222}
\end{eqnarray}

\noindent v) For every $X_{1},X_{2}\in E$

\qquad (a) $\mathcal{R}(\mu (X_{1}),\mu )\wedge \mathcal{R}(\mu (X_{2}),\mu
)=\mathcal{R}(\mu (X_{1})\wedge \mu (X_{2}),\mu )$

\qquad (b) $\mathcal{R}(\mu (X_{1}),\mu )\vee \mathcal{R}(\mu (X_{2}),\mu )=%
\mathcal{R}(\mu (X_{1})\vee \mu (X_{2}),\mu )$

\noindent vi) The map $\mathcal{R}(\mu (X),\mu )$ is quasi-affine with
respect to $X$ in the sense that for every $X_{1},X_{2}\in E$, $\Lambda
\,\in L^{0}(\mathcal{G})$ and $0\leq \Lambda \leq 1,$ we have

$\qquad \mathcal{R}(\mu (\Lambda X_{1}+(1-\Lambda )X_{2}),\mu )\geq \mathcal{%
R}(\mu (X_{1}),\mu )\wedge \mathcal{R}(\mu (X_{2}),\mu )\text{
(quasiconcavity)}$

$\qquad \mathcal{R}(\mu (\Lambda X_{1}+(1-\Lambda )X_{2}),\mu )\leq \mathcal{%
R}(\mu (X_{1}),\mu )\vee \mathcal{R}(\mu (X_{2}),\mu )\text{
(quasiconvexity).}$

\noindent vii) $\inf_{Y\in L^{0}(\mathcal{G})}\mathcal{R}(Y,\mu
_{1})=\inf_{Y\in L^{0}(\mathcal{G})}\mathcal{R}(Y,\mu _{2})$ for every $\mu
_{1},\mu _{2}\in E^{\ast}$.
\end{lemma}

\begin{proof}
i) and ii) follow trivially from the definition.

\noindent iii) The set $\{\pi (\xi )\mid \xi \in E,\;\mu (\xi )\geq Y\}$ is
clearly downward directed. Thus there exists a sequence $\left\{ \xi
_{m}^{\mu }\right\} _{m=1}^{\infty }\in E$ such that 
\begin{equation*}
\mu (\xi _{m}^{\mu })\geq Y\quad \forall \,m\geq 1,\quad \pi (\xi _{m}^{\mu
})\downarrow \mathcal{R}(Y,\mu )\quad \text{as }m\uparrow \infty .
\end{equation*}%
Now let $\mathcal{R}(Y,\mu )<\alpha $: consider the sets $F_{m}=\{\pi (\xi
_{m}^{\mu })<\alpha \}$ and the partition of $\Omega $ given by $G_{1}=F_{1}$
and $G_{m}=F_{m}\setminus G_{m-1}$. Since we assume that $E$ satisfies (CSet) and from the property (REG) we get: 
\begin{equation*}
\xi =\sum_{m=1}^{\infty }\xi _{m}^{\mu }\mathbf{1}_{G_{m}}\in E,\quad \mu
(\xi )\geq Y\text{ and }\pi (\xi )<\alpha .
\end{equation*}

\noindent iv), v) and vi) follow as in \cite{FM09}.

\noindent (vii) Notice that $\mathcal{R}(Y,\mu )\geq \inf_{\xi \in E}\pi
(\xi )$, $\forall \,Y\in L_{\mathcal{F}}^{0}$, implies: $\inf_{Y\in L^{0}(%
\mathcal{G})}\mathcal{R}(Y,\mu )\geq \inf_{\xi \in E}\pi (\xi ).$ On the
other hand, $\pi (\xi )\geq \mathcal{R}(\mu (\xi ),\mu )\geq \inf_{Y\in
L^{0}(\mathcal{G})}\mathcal{R}(Y,\mu )$, $\forall \,\xi \in E$, implies: $%
\inf_{Y\in L^{0}(\mathcal{G})}\mathcal{R}(Y,\mu )\leq \inf_{\xi \in E}\pi
(\xi ).$
\end{proof}

\begin{proof}[Proof of Theorem \protect\ref{EvQco}]
Let $\pi :E\rightarrow \bar{L}^{0}(\mathcal{G})$. There might exist a set $%
A\in \mathcal{G}$ on which the map $\pi $ is constant, in the sense that $%
\pi (\xi )\mathbf{1}_{A}=\pi (\eta )\mathbf{1}_{A}$ for every $\xi ,\eta \in
E$. For this reason we introduce 
\begin{equation*}
\mathcal{A}:=\{B\in \mathcal{G}\mid \pi (\xi )\mathbf{1}_{B}=\pi (\eta )%
\mathbf{1}_{B}\;\forall \,\xi ,\eta \in E\}.
\end{equation*}%
Applying Lemma \ref{L10} in Appendix with $F:=\left\{ \pi (\xi )-\pi (\eta
)\mid \xi ,\eta \in E\right\} $ (we consider the convention $+\infty -\infty
=0$) and $Y_{0}=0$ we can deduce the existence of two maximal sets $A\in 
\mathcal{G}$ and $A^{\vdash}\in \mathcal{G}$ for which $P(A\cap
A^{\vdash})=0 $, $P(A\cup A^{\vdash})=1$ and 
\begin{eqnarray}
\pi (\xi )=\pi (\eta ) &\text{ on }A&\text{for every }\xi ,\eta \in E, 
\notag \\
\pi (\zeta _{1})<\pi (\zeta _{2}) &\text{ on }A^{\vdash}&\text{ for some }%
\zeta _{1},\zeta _{2}\in E.  \label{888}
\end{eqnarray}%
Recall that $\Upsilon_{\pi}\in \mathcal{G} $ is the maximal set on which $%
\pi(\xi)\mathbf{1}_{\Upsilon_{\pi}}=+\infty\mathbf{1}_{\Upsilon_{\pi}}$ for
every $\xi\in E$ and $T_{\pi}$ its complement. Notice that $%
\Upsilon_{\pi}\subset A$. \newline
Fix $x\in E$ and $G=\{\pi (x)<+\infty \}$. For every $\varepsilon \in
L_{++}^{0}(\mathcal{G})$ we set 
\begin{equation}  \label{Yeps}
Y_{\varepsilon }=:0\mathbf{1}_{\Upsilon_{\pi}}+\pi(x)\mathbf{1}_{A\setminus
\Upsilon_{\pi}}+(\pi (x)-\varepsilon )\mathbf{1}_{G\cap
A^{\vdash}}+\varepsilon \mathbf{1}_{G^{C}\cap A^{\vdash}}
\end{equation}
and for every $\varepsilon\in L^0(\mathcal{G} )_{++}$ we set 
\begin{eqnarray}  \label{levelset}
\mathcal{C}_{\varepsilon}=\{\xi\in E\mid \pi(\xi)\mathbf{1}%
_{T_{\varepsilon}}\leq Y_{\varepsilon}\}.
\end{eqnarray}

\noindent\emph{Step 1:} on the set $A$, $\pi (x)=\mathcal{R}(\langle
x,x^{\prime }\rangle, x^{\prime })$ for any $x^{\prime} \in E^{\prime} $ and the
representation 
\begin{equation}
\pi (x)\mathbf{1}_{A}=\max_{x^{\prime }\in E^{\prime }}\mathcal{R}(\langle
x,x^{\prime }\rangle,x^{\prime })\mathbf{1}_{A}  \label{rapprDC}
\end{equation}
trivially holds true on $A$.

\bigskip

\noindent\emph{Step 2:} by the definition of $Y_{\varepsilon}$ we deduce
that if $\mathcal{C}_{\varepsilon}=\emptyset$ for every $\varepsilon\in
L^0_{++}$ then $\pi (x)\leq \pi(\xi)$ on the set $A^{\vdash}$ for every $%
\xi\in E$ and $\pi(x)\mathbf{1}_{A^{\vdash}}=\mathcal{R}(\langle x,x^{\prime
}\rangle,x^{\prime })\mathbf{1}_{A^{\vdash}}$ for any $x^{\prime} $.
The representation 
\begin{equation}
\pi (x)\mathbf{1}_{A^{\vdash}}=\max_{x^{\prime }\in E^{\prime }}\mathcal{R}%
(\langle x,x^{\prime }\rangle, x^{\prime })\mathbf{1}_{A^{\vdash}}
\label{rapprDC1}
\end{equation}
trivially holds true on $A^{\vdash}$. The thesis follows pasting together
equations (\ref{rapprDC}) and (\ref{rapprDC1})

\bigskip

\noindent\emph{Step 3:} we now suppose that there exists $\varepsilon\in
L^0_{++}$ such that $\mathcal{C}_{\varepsilon}\neq\emptyset$. The definition
of $Y_{\varepsilon}$ implies that $\mathcal{C}_{\varepsilon}\mathbf{1}_{A}=E%
\mathbf{1}_{A}$ and $A$ is the maximal element i.e. $A=A_{\mathcal{C}%
_{\varepsilon}}$ (given by Definition \ref{trivialcomponent}). Moreover this
set is conditionally evenly convex and $x$ is outside ${C}_{\varepsilon}$.
The definition of evenly convex set guarantees that there exists $x^{\prime
}_{\varepsilon}\in E^{\prime}$ such that 
\begin{equation}
\langle x,x^{\prime }_{\varepsilon}\rangle> \langle \xi,x^{\prime
}_{\varepsilon}\rangle\quad \text{ on } D_{\mathcal{C}_{\varepsilon}}=A^{%
\vdash},\; \forall\xi\in\mathcal{C}_{\varepsilon}.
\end{equation}
\emph{Claim:} 
\begin{equation}
\{\xi \in E\mid \langle x,x^{\prime }_{\varepsilon}\rangle\mathbf{1}%
_{A^{\vdash}}\leq \langle \xi,x^{\prime }_{\varepsilon}\rangle\mathbf{1}%
_{A^{\vdash}}\}\subseteq \{\xi \in E\mid \pi (\xi )>(\pi (x)-\varepsilon )%
\mathbf{1}_{G}+\varepsilon \mathbf{1}_{G^{C}}\text{ on }A^{\vdash}\}\text{.}
\label{666}
\end{equation}%
In order to prove the claim take $\xi\in E$ such that $\langle x,x^{\prime
}_{\varepsilon}\rangle\mathbf{1}_{A^{\vdash}}\leq \langle \xi,x^{\prime
}_{\varepsilon}\rangle\mathbf{1}_{A^{\vdash}}$. By contra we suppose that
there exists a $F\subset A^{\vdash}$, $F\in\mathcal{G} $ and $\mathbb{P}
(F)>0$ such that $\pi (\xi )\mathbf{1}_F\leq(\pi (x)-\varepsilon )\mathbf{1}%
_{G\cap F}+\varepsilon \mathbf{1}_{G^{C}\cap F}$. Take $\eta\in\mathcal{C}%
_{\varepsilon}$ and define $\overline{\xi}=\eta\mathbf{1}_{F^C}+\xi\mathbf{1}%
_{F}\in\mathcal{C}_{\varepsilon}$ so that we conclude that $\langle
x,x^{\prime }_{\varepsilon}\rangle>\langle \overline{\xi},x^{\prime
}_{\varepsilon}\rangle$ on $A^{\vdash}$. Since $\langle \overline{\xi}%
,x^{\prime }_{\varepsilon}\rangle=\langle \xi,x^{\prime
}_{\varepsilon}\rangle$ on $F$ we reach a contradiction.

\bigskip

Once the claim is proved we end the argument observing that 
\begin{eqnarray}
\pi (x)\mathbf{1}_{A^{\vdash}} &\geq &\sup_{x^{\prime }\in E^{\prime}}\mathcal{R}%
(\langle x,x^{\prime }\rangle,x^{\prime })\mathbf{1}_{A^{\vdash}}=\mathcal{R}%
(\langle x,x^{\prime }_{\varepsilon}\rangle,x^{\prime }_{\varepsilon} )%
\mathbf{1}_{A^{\vdash}}  \label{111} \\
&=&\inf_{\xi \in E}\{\pi (\xi )\mathbf{1}_{A^{\vdash}}\mid \langle
x,x^{\prime }_{\varepsilon}\rangle\mathbf{1}_{A}\leq \langle \xi,x^{\prime
}_{\varepsilon}\rangle\mathbf{1}_{A^{\vdash}}\}  \notag \\
&\geq &\inf_{\xi \in E}\{\pi (\xi )\mathbf{1}_{A^{\vdash}}\mid \pi (\xi
)>(\pi (x)-\varepsilon )\mathbf{1}_{G}+\varepsilon \mathbf{1}_{G^{C}}\text{
on }A^{\vdash}\}  \notag \\
&\geq &(\pi (x)-\varepsilon )\mathbf{1}_{G\cap A^{\vdash}}+\varepsilon 
\mathbf{1}_{G^{C}\cap A^{\vdash}},  \label{sequence}
\end{eqnarray}%
The representation (\ref{rapprPiSup}) follows by taking $\varepsilon $
arbitrary small on $G\cap A^{\vdash}$ and arbitrary big on $G^{C}\cap
A^{\vdash}$ and pasting together the result with equation (\ref{rapprDC}).
\end{proof}

\noindent

\begin{proof}[Proof of Proposition \protect\ref{USC}]
Fix $X\in E$ and consider the classes of sets%
\begin{eqnarray*}
\mathcal{A}:= &&\left\{ B\in \mathcal{G}\mid \forall \xi \in E\text{ }\pi
(\xi )\geq \pi (X)\text{ on }B\right\} , \\
\mathcal{A}^{\vdash }:= &&\{B\in \mathcal{G}\mid \exists \xi \in E\text{
s.t. }\pi (\xi )<\pi (X)\text{ on }B\}.
\end{eqnarray*}%
Then $\mathcal{A}=\left\{ B\in \mathcal{G}\mid \forall Y\in F\text{ }Y\geq
Y_{0}\text{ on }B\right\} $, where $F:=\left\{ \pi (\xi )\mid \xi \in
E\right\} $ and $Y_{0}=\pi (X)$. Applying Lemma \ref{L10}, there exist two
maximal elements $A\in \mathcal{A}$ and $A^{\vdash }\in \mathcal{A}^{\vdash
} $ so that: $P(A\cup A^{\vdash })=1,$ $P(A\cap A^{\vdash })=0,$ 
\begin{eqnarray*}
\pi (\xi )\geq \pi (X)\text{ on }A\text{ for every }\xi \in E &\text{and}%
&\exists \overline{\xi }\in E\text{ s.t. }\pi (\overline{\xi })<\pi (X)\text{
on }A^{\vdash }.
\end{eqnarray*}%
Clearly for every $\mu \in E^{\ast}$. 
\begin{equation}
\pi (X)\mathbf{1}_{A}\geq \mathcal{R}(\mu (X),\mu )\mathbf{1}_{A}\geq \pi (X)%
\mathbf{1}_{A}.  \label{777}
\end{equation}%
Consider $\delta\in L^0_{++}(\mathcal{G} )$. The set 
\begin{equation*}
\mathcal{O}:=\{\xi \in E\mid \pi (\xi )\mathbf{1}_{T_{\pi}}<\pi (X)\mathbf{1}%
_{A^{\vdash }}+(\pi(X)+\delta)\mathbf{1}_{A}\}
\end{equation*}
is open, $L^{0}(\mathcal{G})$-convex (from Remark \ref{remA} ii) and not
empty. Clearly $X\notin \mathcal{O}$ and $\mathcal{O}$ satisfies (CSet). We
thus can apply Theorem 3.15 in \cite{Guo} and find $\mu _{\ast }\in E^{\ast}$
so that 
\begin{equation*}
\mu _{\ast }(X)>\mu _{\ast }(\xi )\quad \text{on } H(\{X\},\mathcal{O}%
),\;\forall \,\xi \in \mathcal{O}.
\end{equation*}
Notice that $\mathbb{P} (H(\{X\},\mathcal{O})\setminus A^{\vdash})=0$. We
apply the argument in Step 3 of the proof of Theorem \ref{EvQco} to find
that 
\begin{equation*}
\{\xi \in E\mid \mu _{\ast }(X)\mathbf{1}_{A^{\vdash }}\leq \mu _{\ast }(\xi
)\mathbf{1}_{A^{\vdash }}\}\subseteq \{\xi \in E\mid \pi (\xi )\mathbf{1}%
_{A^{\vdash }}\geq \pi (X)\mathbf{1}_{A^{\vdash }}\}.
\end{equation*}%
From (\ref{TakingOutM})-(\ref{222}) we derive 
\begin{eqnarray*}
\pi (X)\mathbf{1}_{A^{\vdash }} &\geq &\mathcal{R}(\mu _{\ast }(X),\mu
_{\ast })\mathbf{1}_{A^{\vdash }}=\inf_{\xi \in E}\{\pi (\xi )\mathbf{1}%
_{A^{\vdash }}\mid \mu _{\ast }(X)\mathbf{1}_{A^{\vdash }}\leq \mu _{\ast
}(\xi )\mathbf{1}_{A^{\vdash }}\} \\
&\geq &\inf_{\xi \in E}\{\pi (\xi )\mathbf{1}_{A^{\vdash }}\mid \pi (\xi )%
\mathbf{1}_{A^{\vdash }}\geq \pi (X)\mathbf{1}_{A_{M}^{\vdash }}\}\geq \pi
(X)\mathbf{1}_{A_{M}^{\vdash }}.
\end{eqnarray*}%
The thesis then follows from (\ref{777}).
\end{proof}

\section{Appendix}

\begin{lemma}
\label{external}For any sets $\mathcal{C}\subseteq E$ and $\mathcal{D}%
\subseteq E$ set: 
\begin{eqnarray*}
\mathcal{A} &=&\left\{ B\in \mathcal{G}\mid \forall y\in \mathcal{D}^{cc}\;
\exists \xi \in \mathcal{C}^{cc}\text{ s.t. }\mathbf{1}_{B}y=\mathbf{1}%
_{B}\xi \right\} , \\
\mathcal{A}^{\vdash } &=&\left\{ B\in \mathcal{G}\mid \exists y\in \mathcal{D%
}^{cc}\text{ s.t. }y\text{ is outside}\mid _{B}\mathcal{C}^{cc}\right\} .
\end{eqnarray*}%
Then there exist the maximal set $A_{M}\in \mathcal{A}$ of $\mathcal{A}$ and
the maximal set $A_{M}^{\vdash }\in \mathcal{A}^{\vdash }$ of $\mathcal{A}%
^{\vdash }$, one of which may have zero probability, that satisfy 
\begin{equation*}
\forall y\in \mathcal{D}^{cc}\text{ }\exists \xi \in \mathcal{C}^{cc}\text{
s.t. }\mathbf{1}_{A_{M}}y=\mathbf{1}_{A_{M}}\xi 
\end{equation*}%
\begin{equation*}
\exists y\in \mathcal{D}^{cc}\text{ s.t. }y\text{ is outside}\mid
_{A_{M}^{\vdash }}\mathcal{C}^{cc}, 
\end{equation*}%
and $A_{M}^{\vdash }$ is the $\mathbb{P}$-a.s unique complement of $A_{M}$.

Suppose in addition that $\mathcal{D}=E$ and $\mathcal{C}=\mathcal{C}^{cc}.$
Then the class $\mathcal{A}$ coincides with the class $\mathcal{A}(\mathcal{C%
})=\left\{ B\in \mathcal{G}\mid \mathbf{1}_{A}E=\mathbf{1}_{A}\mathcal{C}%
\right\} $ introduced in the Notation \ref{trivialcomponent}. Henceforth:
the maximal set of $\mathcal{A}(\mathcal{C})$ is $A_{\mathcal{C}}=A_{M};$ $%
D_{\mathcal{C}}=A_{M}^{\vdash };$ $\mathbf{1}_{A_{\mathcal{C}}}E=\mathbf{1}%
_{A_{\mathcal{C}}}\mathcal{C}$; and there exists $y\in E$ s.t. $y$ is
outside $\mathcal{C}$. If $x\notin \mathcal{C}$ then $\mathbb{P}(D_{\mathcal{%
C}})>0$ and $\chi =\{y\in E\mid y$ is outside $\mathcal{C}\}\neq \varnothing 
$.
\end{lemma}

\begin{proof}
The two classes $\mathcal{A}$ and $\mathcal{A}^{\vdash }$ are closed with
respect to \emph{countable} union. Indeed, for the family $\mathcal{A}%
^{\vdash }$, suppose that $B_{i}\in \mathcal{A}^{\vdash },$ $y_{i}\in 
\mathcal{D}^{cc}$ s.t. $y_{i}$ is outside$\mid _{B_{i}}\mathcal{C}^{cc}.$
Define $\widetilde{B}_{1}:=B_{1}$, $\widetilde{B}_{i}:=B_{i}\setminus
B_{i-1} $, $B:=\bigcup\limits_{i=1}^{\infty }\widetilde{B}%
_{i}=\bigcup\limits_{i=1}^{\infty }B_{i}$. Then $y_{i}$ is outside$\mid _{%
\widetilde{B}_{i}}\mathcal{C}^{cc}$, $\widetilde{B}_{i}$ are disjoint
elements of $\mathcal{A}^{\vdash }$ and $y^{\ast }:=\sum_{1}^{\infty
}y_{i}1_{\widetilde{B}_{i}}\in \mathcal{D}^{cc}$. Since $y_{i}1_{\widetilde{B%
}_{i}}=y^{\ast }1_{\widetilde{B}_{i}}$, $y$ is outside$\mid _{\widetilde{B}%
_{i}}\mathcal{C}^{cc} $ for all $i$ and so $y$ is outside$\mid _{B}\mathcal{C%
}^{cc}$. Thus $B\in \mathcal{A}^{\vdash }$. Similarly for the class $%
\mathcal{A}$.

The Remark \ref{remMAX} guarantees the existence of the two maximal sets $%
A_{M}\in \mathcal{A}$ and $A_{M}^{\vdash }\in \mathcal{A}^{\vdash }$, so
that: $B\in \mathcal{A}$ implies $B\subseteq A_{M}$; $B^{\vdash }\in 
\mathcal{A}^{\vdash }$ implies $B^{\vdash }\subseteq A_{M}^{\vdash }$.

Obviously, $P(A_{M}\cap A_{M}^{\vdash })=0,$ as $A_{M}\in \mathcal{A}$ and $%
A_{M}^{\vdash }\in \mathcal{A}^{\vdash }$. We claim that 
\begin{equation}
P(A_{M}\cup A_{M}^{\vdash })=1.  \label{bb}
\end{equation}%
To show (\ref{bb}) let $D:=\Omega \setminus \left\{ A_{M}\cup A_{M}^{\vdash
}\right\} \in \mathcal{G}$. By contradiction suppose that $\mathbb{P}(D)>0$.
From $D\subseteq (A_{M})^{C}$ and the maximality of $A_{M}$ we get $D\notin 
\mathcal{A}$. This implies that there exists $y\in \mathcal{D}^{cc}$ such
that 
\begin{equation}
\mathbf{1}_{D}\left\{ y\right\} \cap \mathbf{1}_{D}\mathcal{C}%
^{cc}=\varnothing  \label{1d}
\end{equation}%
and obviously $y\notin \mathcal{C}^{cc}$, as $\mathbb{P}(D)>0$. By the Lemma %
\ref{LemmaGuoSet} there exists a set $H_{\mathcal{C}^{cc},y}:=H\in \mathcal{G%
}$ satisfying $\mathbb{P}(H)>0$, (\ref{11}) and (\ref{12}) with $\mathcal{C}$
replaced by $\mathcal{C}^{cc}.$

Condition (\ref{12}) implies that $H\in \mathcal{A}^{\vdash }$ and then $%
H\subseteq A_{M}^{\vdash }$. From (\ref{11}) we deduce that there exists $%
\xi \in \mathcal{C}^{cc}$ s.t. $\mathbf{1}_{A}y=\mathbf{1}_{A}\xi $ for all $%
A\subseteq \Omega \backslash H.$ Then (\ref{1d}) implies that $D$ is not
contained in $\Omega \backslash H,$ so that: $\mathbb{P}(D\cap H)>0.$ This
is a contradiction since $D\cap H\subseteq D\subseteq (A_{M}^{\vdash })^{C}$
and $D\cap H\subseteq H\subseteq A_{M}^{\vdash }$.

Item 1 is a trivial consequence of the definitions.
\end{proof}

\begin{lemma}
\label{L10}With the symbol $\trianglerighteq $ denote any one of the binary
relations $\geq ,$ $\leq ,$ $=,$ $>$, $<$ and with $\vartriangleleft $ its
negation. Consider a class $F\subseteq \bar{L}^{0}(\mathcal{G})$ of random
variables, $Y_{0}\in \bar{L}^{0}(\mathcal{G})$ and the classes of sets 
\begin{eqnarray*}
\mathcal{A}:= &\{A\in &\mathcal{G}\mid \forall \,Y\in F\text{ }%
Y\trianglerighteq Y_{0}\text{ on }A\}, \\
\mathcal{A}^{\vdash }:= &\{A^{\vdash }\in &\mathcal{G}\mid \exists \,Y\in F%
\text{ s.t. }Y\vartriangleleft Y_{0}\text{ on }A^{\vdash }\}.
\end{eqnarray*}%
Suppose that for any sequence of disjoint sets $A_{i}^{\vdash }\in \mathcal{A%
}^{\vdash }$ and the associated r.v. $Y_{i}\in F$ we have $\sum_{1}^{\infty
}Y_{i}1_{A_{i}^{\vdash }}\in F$. Then there exist two maximal sets $A_{M}\in 
\mathcal{A}$ and $A_{M}^{\vdash }\in \mathcal{A}^{\vdash }\mathcal{\ }$such
that $P(A_{M}\cap A_{M}^{\vdash })=0$, $P(A_{M}\cup A_{M}^{\vdash })=1$ and 
\begin{eqnarray*}
&&Y\trianglerighteq Y_{0}\text{ on }A_{M}\text{, }\forall Y\in F\text{ } \\
&&\overline{Y}\vartriangleleft Y_{0}\text{ on }A_{M}^{\vdash }\text{, for
some }\overline{Y}\in F.
\end{eqnarray*}
\end{lemma}

\begin{proof}
Notice that $\mathcal{A}$ and $\mathcal{A}^{\vdash }$ are closed with
respect to \emph{countable} union. This claim is obvious for $\mathcal{A}$.
For $\mathcal{A}^{\vdash }$, suppose that $A_{i}^{\vdash }\in \mathcal{A}%
^{\vdash }$ and that $Y_{i}\in F$ satisfies $P(\left\{ Y_{i}\vartriangleleft
Y_{0}\right\} \cap A_{i}^{\vdash })=P(A_{i}^{\vdash }).$ Defining $%
B_{1}:=A_{i}^{\vdash }$, $B_{i}:=A_{i}^{\vdash }\setminus B_{i-1}$, $%
A_{\infty }^{\vdash }:=\bigcup\limits_{i=1}^{\infty }A_{i}^{\vdash
}=\bigcup\limits_{i=1}^{\infty }B_{i}$ we see that $B_{i}$ are disjoint
elements of $\mathcal{A}^{\vdash }$ and that $Y^{\ast }:=\sum_{1}^{\infty
}Y_{i}1_{B_{i}}\in F$ satisfies $P(\left\{ Y^{\ast }\vartriangleleft
Y_{0}\right\} \cap A_{\infty }^{\vdash })=P(A_{\infty }^{\vdash })$ and so $%
A_{\infty }^{\vdash }\in \mathcal{A}^{\vdash }$.

The Remark \ref{remMAX} guarantees the existence of two sets $A_{M}\in 
\mathcal{A}$ and $A_{M}^{\vdash }\in \mathcal{A}^{\vdash }$ such that:

(a) $P(A\cap (A_{M})^{C})=0$ for all $A\in \mathcal{A}$,

(b) $P(A^{\vdash }\cap (A_{M}^{\vdash })^{C})=0$ for all $A^{\vdash }\in 
\mathcal{A}^{\vdash }$.

Obviously, $P(A_{M}\cap A_{M}^{\vdash })=0,$ as $A_{M}\in \mathcal{A}$ and $%
A_{M}^{\vdash }\in \mathcal{A}^{\vdash }$. To show that $P(A_{M}\cup
A_{M}^{\vdash })=1$, let $D:=\Omega \setminus \left\{ A_{M}\cup
A_{M}^{\vdash }\right\} \in \mathcal{G}$. By contradiction suppose that $%
P(D)>0$. As $D\subseteq (A_{M})^{C}$, from condition (a) we get $D\notin 
\mathcal{A}$. Therefore, $\exists \overline{Y}\in F$ s.t. $P(\left\{ 
\overline{Y}\trianglerighteq Y_{0}\right\} \cap D)<P(D),$ i.e. $P(\left\{ 
\overline{Y}\vartriangleleft Y_{0}\right\} \cap D)>0$. If we set $B:=\left\{ 
\overline{Y}\vartriangleleft Y_{0}\right\} \cap D$ then it satisfies $%
P(\left\{ \overline{Y}\vartriangleleft Y_{0}\right\} \cap B)=P(B)>0$ and, by
definition of $\mathcal{A}^{\vdash }$, $B$ belongs to $\mathcal{A}^{\vdash }$%
. On the other hand, as $B\subseteq D\subseteq (A_{M}^{\vdash })^{C}$, $%
P(B)=P(B\cap (A_{M}^{\vdash })^{C}),$ and from condition (b) $P(B\cap
(A_{M}^{\vdash })^{C})=0$, which contradicts $P(B)>0$.
\end{proof}

\end{document}